\documentclass[journal]{IEEEtran}

\usepackage{indentfirst}
\usepackage[T1]{fontenc}
\usepackage[]{graphicx}
\usepackage[]{epsfig}
\usepackage{amsmath}
\usepackage{amssymb}
\usepackage{verbatim}
\usepackage{tikz} 
\usepackage{bbm}
\usepackage{enumerate}
\usepackage{pdfsync}
\usepackage{algorithmic}
\usepackage{algorithm}
\usepackage{array}
\usepackage{mathrsfs}
\usepackage{dsfont}
\usetikzlibrary{calc} 
\usetikzlibrary{snakes,arrows}

\usepackage[linktoc=page,colorlinks=true,linkcolor=blue,citecolor=blue]{hyperref}
\usepackage{bookmark}
\usepackage{makeidx}
\usepackage{amsfonts}
\usepackage{amsthm}
\usepackage{url}
\usepackage[caption=false,font=footnotesize]{subfig}
\usepackage{cite}
\usepackage{color}
\usepackage{tocvsec2}

\newtheorem{assumption}{Assumption}
\newtheorem{lemma}{Lemma}
\newtheorem{proposition}{Proposition}
\newtheorem{theorem}{Theorem}
\newtheorem{definition}{Definition}
\newtheorem{corollary}[theorem]{Corollary}
\newtheorem*{remark}{Remark}

\newcommand{\E}[1]{\mathbf{E}_{}\left[#1\right]}
\newcommand{\var}[1]{{\mathbf{V}}_{}\left[#1\right]}

\def\P{{\mathbf{P}}}

\definecolor{int}{rgb}{0.8,0.8,1}
\definecolor{ext}{rgb}{0,0,1}
\definecolor{redbox}{RGB}{191,18,56} % 
\definecolor{blackbox}{RGB}{0,0,0} % 
\definecolor{brownbox}{RGB}{128,99,90} %

\definecolor{clair}{rgb}{0.910,0.933,0.957}
\definecolor{moyen}{rgb}{0.102,0.349,0.557}
\definecolor{fonce}{rgb}{0.067,0.231,0.369}
\definecolor{titres}{rgb}{0.137,0.466,0.741}
\definecolor{apricot}{rgb}{0.961,0.506,0.216}
\definecolor{grey}{rgb}{0.58,0.58,0.58}

\begin{document}

\pgfdeclarelayer{background layer} 
\pgfdeclarelayer{foreground layer} 
\pgfsetlayers{background layer,main,foreground layer}

%\title{On the Complexity of a Reduction Algorithm for Random Abstract
%Simplicial Complexes}
\title{Random Abstract Simplicial Complexes Reduction}
\author{\IEEEauthorblockN{Ana\"\i s Vergne,
Laurent Decreusefond, and
Philippe Martins,~\IEEEmembership{Senior Member,~IEEE}}\\
\IEEEauthorblockA{LTCI, T\'el\'ecom ParisTech, Universit\'e Paris-Saclay, 75013, Paris, France}
\thanks{Manuscript created on January 13, 2017.}}

\maketitle
\begin{abstract}
  Random abstract simplicial complex representation provides a
  mathematical description of wireless networks and their topology. In
  order to reduce the energy consumption in this type of network, we
  intend to reduce the number of network nodes without modifying
  neither the connectivity nor the coverage of the network. In this
  paper, we present a reduction algorithm that lower the number of
  points of an abstract simplicial complex in an optimal order while
  maintaining its topology. Then, we study the complexity of such an
  algorithm for a network simulated by a binomial point process and
  represented by a Vietoris-Rips complex.
\end{abstract}


\section{Introduction}
Wireless networks are everyday more present in our lifes: WiFi is the
main internet access in our homes, cellular systems such as 4G and
soon 5G provide its access everywhere else. Moreover with IoT, every
object in our kitchen or in our bathroom will in the near future be
connected as well. The quality of service of this type of network is
primarily its connectivity and its coverage, only after checking this
first two characteristics can come the capacity of the network, that
is the number of users or connected devices a network can accept.
However whether a network of sensors is fully connected, or whether a
set of base stations does cover a whole domain, is not that easy to
determine since network nodes are often irregularly deployed. Indeed,
recent works such as~\cite{deng_ginibre_2015} or
\cite{gomez_case_2015} show that cellular networks deployment can be 
approached by random point processes going from the repulsive Ginibre
point process to the neither repulsive nor attractive Poisson point
process depending on the type of area (rural or urban) or the type of
systems (every systems or only 4G systems for example).

Algebraic topology,~\cite{hatcher_algebraic_2002}, turns out to
provide the solutions to the problem of how to compute
the topology of a random set of points. Based on the geometrical data
of the network (network nodes locations, communication or coverage
radii), it is possible to build a combinatorial object representing
it: the simplicial complex. Basically a simplicial complex is the generalization of the
concept of graph, it is made of $k$-simplices where $0$-simplices are vertices,
$1$-simplices are edges, $2$-simplices are triangles, $3$-simplices are
tetrahedron and so on. In particular, the \u{C}ech simplicial complex allows to
represent exactly the coverage of the union of the coverage disks as stated in
the Nerve lemma in~\cite{ghrist_coverage_2005}. Then algebraic
topology is a tool to compute the number of connected components, of
coverage holes, and of 3D voids, that are the so-called Betti numbers
of the simplicial complex representing the network, as detailed in
\cite{de_silva_coordinate-free_2006}. However, the computational time
to obtain the Betti numbers can explode with the size of the
simplicial complex, then it is possible to compute them in a
decentralized way as seen in~\cite{muhammad_decentralized_2007}, or
using persistent homology in~\cite{zomorodian_computing_2005 ,
  silva_coverage_2007}. 

In this paper, we present a reduction algorithm for abstract
simplicial complexes. Points, called $0$-simplices, are removed
one-by-one from an abstract simplicial complex while its topology
remain unchanged. The removal order in the algorithm is optimal for
the complexity of the abstract simplicial complex implementation and
of the algorithm. We use the Vietoris-Rips simplicial complex
built on a binomial point process for representing a wireless network
and illustrate our algorithm. A first version of this algorithm has
been presented in~\cite{vergne_reduction_2013} and it has been used
for cellular networks applications in~\cite{vergne_simplicial_2015 ,
  yan_homology-based_2015, le_simplicial_2015}. We investigate the
complexity of our reduction algorithm and show that it depends on the
size of the largest simplex of the abstract simplicial complex. We
compute its almost sure asymptotical behavior for a Vietoris-Rips
complex based on a binomial point process, in which case it is also
known as the clique number in a random geometric graph. 

This is the first reduction algorithm for abstract simplicial complexes that
uses homology to reduce the complex that we know of. Usually reduction
algorithms for simplicial complexes are used to reduce complexes
prior to the computation of their topology in order to reduce its complexity.
For example, in~\cite{dlotko_distributed_2012} and
\cite{kaczynski_homology_1998}, the authors use reduction of chain 
complexes in order to compute the homology groups and the Betti
numbers. Witness complexes of~\cite{de_silva_topological_2004} are
another example of simplicial complexes reduction: the 
simplicial complex is reduced to a given number of vertices in order
to compute the various topological invariants, such as the Betti
numbers. So reduction of a simplicial complex has been used in order
to compute its topology, we intend to do the opposite: reduction of
the simplicial complex becomes the aim, while the homology computation
is the mean to do it. The reduction problem can also be seen as a
dominating graph problem~\cite{haynes_fundamentals_1998}. However,
since there is no notion of coverage in graphs, algorithms for the
dominating graph problem do not maintain the topology of the initial
simplicial complex. Our problem has also been studied under a
game-theoretic approach in~\cite{campos-nanez_game-theoretic_2008},
where the authors define a coverage function. But they can only
identify sub-optimal solutions that do not guarantee an unmodified
coverage.

 When computing the complexity of the algorithm, we focus on
 Vietoris-Rips complexes based on binomial point processes that fall
 into the class of random geometric complexes. There exists known
 results for this class of complexes~\cite{bobrowski_topology_2014},
 especially the moments of the number of $k$-simplices are explicitly
 known~\cite{decreusefond_simplicial_2014}. In the end, we are reduced
 to compute the behavior of the size of the largest simplex, which is
 known as the clique number in graph theory. The clique number of the
 random geometric graph has been heavily studied in the literature, and its
behavior described according to percolation regimes in
\cite{penrose_random_2003}. In~\cite{goel_monotone_2005}, it is
proved that monotone properties of random geometric graph have sharp
thresholds. Hence, in~\cite{penrose_focusing_2002 ,
  muller_two-point_2008}, the authors prove that, in the subcritical
regime, the clique number becomes concentrated on two consecutive
integers. Moreover, in the subcritical regime, weak laws of large
numbers~\cite{penrose_weak_2003} and central limits theorems
\cite{penrose_central_2001} have been found for some functionals,
including the clique number. Then for the supercritical regime,
in\cite{appel_maximum_1997}, the authors described the behavior of the clique
number. In this paper, we intend to gather all these results in one
place: we provide the almost sure asymptotical behavior of the clique
number of the random geometric graph for every of the three regimes
thanks to the exact formulas of the moments of the number of
$k$-simplices computed in~\cite{decreusefond_simplicial_2014}.

The remainder of the paper is organized as followed. First in Section
\ref{sec_prel}, we remind some simplicial homology and algebraic
topology definitions and properties. Then the reduction
algorithm and its properties are described in Section
\ref{sec_ra}. The complexity of the algorithm is investigated for a
random set of points in Section~\ref{sec_comp}. Finally we conclude in
Section~\ref{sec_ccl}.

\section{Mathematical background}
\label{sec_prel}
\subsection{Simplicial Homology}
Considering a set of points representing network nodes, the first idea
to apprehend the topology of the network would be to look at the
neighbors graph: if the distance between two points is less than a
given parameter then an edge is drawn between them. An example of a
neighbors graph can be seen in Fig.~\ref{fig_neighbor}. However this
representation is too limited to transpose the network's
topology. First, only $2$-by-$2$ relationships are represented in the
graph, there is no way to grasp interactions between three or more
nodes. Moreover, there is no concept of coverage in a graph. That is
why we are interested in more complex objects.
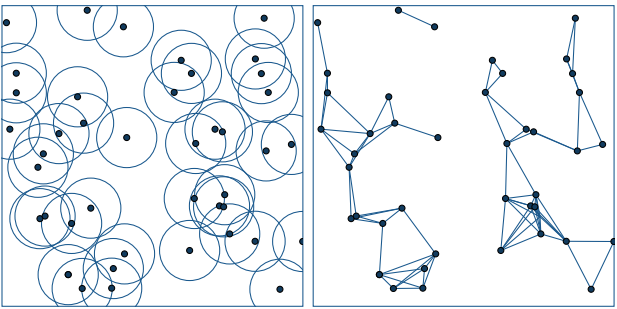
\begin{figure}[h]
  \centering
\begin{tikzpicture}[scale=4]
\draw  [draw=moyen] (0,0) rectangle (1,1); 
\useasboundingbox (0,0) rectangle (1,1);
\coordinate (x1) at (0.9238,0.0566);
\coordinate (x2) at (0.7228,0.3340);
\coordinate (x3) at (0.2711,0.6092);
\coordinate (x4) at (0.0474,0.7752);
\coordinate (x5) at (0.2200,0.1053);
\coordinate (x6) at (0.4072,0.1743);
\coordinate (x7) at (0.2512,0.6968);
\coordinate (x8) at (0.5722,0.7114);
\coordinate (x9) at (0.1261,0.2915);
\coordinate (x10) at (0.8779,0.5168);
\coordinate (x11) at (0.2832,0.9846);
\coordinate (x12) at (0.7325,0.5803);
\coordinate (x13) at (0.7567,0.2407);
\coordinate (x14) at (0.0475,0.7101);
\coordinate (x15) at (0.8411,0.2162);
\coordinate (x16) at (0.9994,0.2155);
\coordinate (x17) at (0.1194,0.4624);
\coordinate (x18) at (0.1374,0.5071);
\coordinate (x19) at (0.8713,0.9582);
\coordinate (x20) at (0.0148,0.9434);
\coordinate (x21) at (0.2313,0.2756);
\coordinate (x22) at (0.3645,0.0602);
\coordinate (x23) at (0.7080,0.5890);
\coordinate (x24) at (0.1893,0.5744);
\coordinate (x25) at (0.4145,0.5612);
\coordinate (x26) at (0.3700,0.1255);
\coordinate (x27) at (0.6392,0.3585);
\coordinate (x28) at (0.2951,0.3264);
\coordinate (x29) at (0.7366,0.3309);
\coordinate (x30) at (0.8620,0.7739);
\coordinate (x31) at (0.9619,0.5384);
\coordinate (x32) at (0.1431,0.3001);
\coordinate (x33) at (0.5956,0.8182);
\coordinate (x34) at (0.0261,0.5891);
\coordinate (x35) at (0.6296,0.7746);
\coordinate (x36) at (0.4035,0.9296);
\coordinate (x37) at (0.8850,0.7112);
\coordinate (x38) at (0.7402,0.3712);
\coordinate (x39) at (0.6438,0.5415);
\coordinate (x40) at (0.6236,0.1856);
\coordinate (x41) at (0.2666,0.0592);
\coordinate (x42) at (0.8421,0.8227);
\coordinate (x43) at (0.7540,0.2975);
\begin{pgfonlayer}{foreground layer}
\clip (0,0) rectangle (1,1); 
\draw [fill=fonce] (x1) circle (0.01cm);
\draw [fill=fonce] (x2) circle (0.01cm);
\draw [fill=fonce] (x3) circle (0.01cm);
\draw [fill=fonce] (x4) circle (0.01cm);
\draw [fill=fonce] (x6) circle (0.01cm);
\draw [fill=fonce] (x9) circle (0.01cm);
\draw [fill=fonce] (x10) circle (0.01cm);
\draw [fill=fonce] (x11) circle (0.01cm);
\draw [fill=fonce] (x14) circle (0.01cm);
\draw [fill=fonce] (x15) circle (0.01cm);
\draw [fill=fonce] (x16) circle (0.01cm);
\draw [fill=fonce] (x17) circle (0.01cm);
\draw [fill=fonce] (x18) circle (0.01cm);
\draw [fill=fonce] (x20) circle (0.01cm);
\draw [fill=fonce] (x21) circle (0.01cm);
\draw [fill=fonce] (x22) circle (0.01cm);
\draw [fill=fonce] (x23) circle (0.01cm);
\draw [fill=fonce] (x24) circle (0.01cm);
\draw [fill=fonce] (x25) circle (0.01cm);
\draw [fill=fonce] (x27) circle (0.01cm);
\draw [fill=fonce] (x28) circle (0.01cm);
\draw [fill=fonce] (x30) circle (0.01cm);
\draw [fill=fonce] (x33) circle (0.01cm);
\draw [fill=fonce] (x35) circle (0.01cm);
\draw [fill=fonce] (x36) circle (0.01cm);
\draw [fill=fonce] (x37) circle (0.01cm);
\draw [fill=fonce] (x38) circle (0.01cm);
\draw [fill=fonce] (x39) circle (0.01cm);
\draw [fill=fonce] (x40) circle (0.01cm);
\draw [fill=fonce] (x41) circle (0.01cm);
\draw [fill=fonce] (x42) circle (0.01cm);
\draw [fill=fonce] (x5) circle (0.01cm);
\draw [fill=fonce] (x5) circle (0.01cm);
\draw [fill=fonce] (x19) circle (0.01cm);
\draw [fill=fonce] (x29) circle (0.01cm);
\draw [fill=fonce] (x13) circle (0.01cm);
\draw [fill=fonce] (x7) circle (0.01cm);
\draw [fill=fonce] (x8) circle (0.01cm);
\draw [fill=fonce] (x31) circle (0.01cm);
\draw [fill=fonce] (x32) circle (0.01cm);
\draw [fill=fonce] (x26) circle (0.01cm);
\draw [fill=fonce] (x34) circle (0.01cm);
\draw [fill=fonce] (x12) circle (0.01cm);
\end{pgfonlayer}{foreground layer}
\begin{pgfonlayer}{background layer}
\clip (0,0) rectangle (1,1); 
\draw [color=moyen] (x1) circle (0.1cm);
\draw [color=moyen]  (x2) circle (0.1cm);
\draw [color=moyen]  (x3) circle (0.1cm);
\draw [color=moyen] (x4) circle (0.1cm);
\draw [color=moyen] (x5) circle (0.1cm);
\draw [color=moyen] (x6) circle (  0.1cm);
\draw [color=moyen]  (x7) circle ( 0.1cm);
\draw [color=moyen] (x8) circle ( 0.1cm);
\draw [color=moyen] (x9) circle ( 0.1cm);
\draw [color=moyen]  (x10) circle ( 0.1cm);
\draw [color=moyen]  (x11) circle ( 0.1cm);
\draw [color=moyen] (x12) circle (0.1cm);
\draw [color=moyen] (x13) circle ( 0.1cm);
\draw [color=moyen] (x14) circle (  0.1cm);
\draw [color=moyen]  (x15) circle ( 0.1cm);
\draw [color=moyen]  (x16) circle ( 0.1cm);
\draw [color=moyen]  (x17) circle ( 0.1cm);
\draw [color=moyen]  (x18) circle ( 0.1cm);
\draw [color=moyen]  (x19) circle ( 0.1cm);
\draw [color=moyen]  (x20) circle ( 0.1cm);
\draw [color=moyen] (x21) circle ( 0.1cm);
\draw [color=moyen] (x22) circle ( 0.1cm);
\draw [color=moyen]  (x23) circle ( 0.1cm);
\draw [color=moyen]  (x24) circle (  0.1cm);
\draw [color=moyen]  (x25) circle (  0.1cm);
\draw [color=moyen] (x26) circle (  0.1cm);
\draw [color=moyen]  (x27) circle ( 0.1cm);
\draw [color=moyen]  (x28) circle (  0.1cm);
\draw [color=moyen] (x29) circle ( 0.1cm);
\draw [color=moyen] (x30) circle ( 0.1cm);
\draw [color=moyen] (x31) circle ( 0.1cm);
\draw [color=moyen] (x32) circle ( 0.1cm);
\draw [color=moyen] (x33) circle ( 0.1cm);
\draw [color=moyen] (x34) circle (  0.1cm);
\draw [color=moyen]  (x35) circle ( 0.1cm);
\draw [color=moyen]  (x36) circle ( 0.1cm);
\draw [color=moyen] (x37) circle ( 0.1cm);
\draw [color=moyen]  (x38) circle ( 0.1cm);
\draw [color=moyen]  (x39) circle ( 0.1cm);
\draw [color=moyen]  (x40) circle ( 0.1cm);
\draw [color=moyen]  (x41) circle ( 0.1cm);
\draw [color=moyen]  (x42) circle ( 0.1cm);
\end{pgfonlayer}{background layer}
\end{tikzpicture} 
\begin{tikzpicture}[scale=4]
\draw  [draw=moyen] (0,0) rectangle (1,1); 
\useasboundingbox (0,0) rectangle (1,1);
\coordinate (x1) at (0.9238,0.0566);
\coordinate (x2) at (0.7228,0.3340);
\coordinate (x3) at (0.2711,0.6092);
\coordinate (x4) at (0.0474,0.7752);
\coordinate (x5) at (0.2200,0.1053);
\coordinate (x6) at (0.4072,0.1743);
\coordinate (x7) at (0.2512,0.6968);
\coordinate (x8) at (0.5722,0.7114);
\coordinate (x9) at (0.1261,0.2915);
\coordinate (x10) at (0.8779,0.5168);
\coordinate (x11) at (0.2832,0.9846);
\coordinate (x12) at (0.7325,0.5803);
\coordinate (x13) at (0.7567,0.2407);
\coordinate (x14) at (0.0475,0.7101);
\coordinate (x15) at (0.8411,0.2162);
\coordinate (x16) at (0.9994,0.2155);
\coordinate (x17) at (0.1194,0.4624);
\coordinate (x18) at (0.1374,0.5071);
\coordinate (x19) at (0.8713,0.9582);
\coordinate (x20) at (0.0148,0.9434);
\coordinate (x21) at (0.2313,0.2756);
\coordinate (x22) at (0.3645,0.0602);
\coordinate (x23) at (0.7080,0.5890);
\coordinate (x24) at (0.1893,0.5744);
\coordinate (x25) at (0.4145,0.5612);
\coordinate (x26) at (0.3700,0.1255);
\coordinate (x27) at (0.6392,0.3585);
\coordinate (x28) at (0.2951,0.3264);
\coordinate (x29) at (0.7366,0.3309);
\coordinate (x30) at (0.8620,0.7739);
\coordinate (x31) at (0.9619,0.5384);
\coordinate (x32) at (0.1431,0.3001);
\coordinate (x33) at (0.5956,0.8182);
\coordinate (x34) at (0.0261,0.5891);
\coordinate (x35) at (0.6296,0.7746);
\coordinate (x36) at (0.4035,0.9296);
\coordinate (x37) at (0.8850,0.7112);
\coordinate (x38) at (0.7402,0.3712);
\coordinate (x39) at (0.6438,0.5415);
\coordinate (x40) at (0.6236,0.1856);
\coordinate (x41) at (0.2666,0.0592);
\coordinate (x42) at (0.8421,0.8227);
\coordinate (x43) at (0.7540,0.2975);
\begin{pgfonlayer}{foreground layer}
\clip (0,0) rectangle (1,1); 
\draw [fill=fonce] (x1) circle (0.01cm);
\draw [fill=fonce] (x2) circle (0.01cm);
\draw [fill=fonce] (x3) circle (0.01cm);
\draw [fill=fonce] (x4) circle (0.01cm);
\draw [fill=fonce] (x6) circle (0.01cm);
\draw [fill=fonce] (x9) circle (0.01cm);
\draw [fill=fonce] (x10) circle (0.01cm);
\draw [fill=fonce] (x11) circle (0.01cm);
\draw [fill=fonce] (x14) circle (0.01cm);
\draw [fill=fonce] (x15) circle (0.01cm);
\draw [fill=fonce] (x16) circle (0.01cm);
\draw [fill=fonce] (x17) circle (0.01cm);
\draw [fill=fonce] (x18) circle (0.01cm);
\draw [fill=fonce] (x20) circle (0.01cm);
\draw [fill=fonce] (x21) circle (0.01cm);
\draw [fill=fonce] (x22) circle (0.01cm);
\draw [fill=fonce] (x23) circle (0.01cm);
\draw [fill=fonce] (x24) circle (0.01cm);
\draw [fill=fonce] (x25) circle (0.01cm);
\draw [fill=fonce] (x27) circle (0.01cm);
\draw [fill=fonce] (x28) circle (0.01cm);
\draw [fill=fonce] (x30) circle (0.01cm);
\draw [fill=fonce] (x33) circle (0.01cm);
\draw [fill=fonce] (x35) circle (0.01cm);
\draw [fill=fonce] (x36) circle (0.01cm);
\draw [fill=fonce] (x37) circle (0.01cm);
\draw [fill=fonce] (x38) circle (0.01cm);
\draw [fill=fonce] (x39) circle (0.01cm);
\draw [fill=fonce] (x40) circle (0.01cm);
\draw [fill=fonce] (x41) circle (0.01cm);
\draw [fill=fonce] (x42) circle (0.01cm);
\draw [fill=fonce] (x5) circle (0.01cm);
\draw [fill=fonce] (x5) circle (0.01cm);
\draw [fill=fonce] (x19) circle (0.01cm);
\draw [fill=fonce] (x29) circle (0.01cm);
\draw [fill=fonce] (x13) circle (0.01cm);
\draw [fill=fonce] (x7) circle (0.01cm);
\draw [fill=fonce] (x8) circle (0.01cm);
\draw [fill=fonce] (x31) circle (0.01cm);
\draw [fill=fonce] (x32) circle (0.01cm);
\draw [fill=fonce] (x26) circle (0.01cm);
\draw [fill=fonce] (x34) circle (0.01cm);
\draw [fill=fonce] (x12) circle (0.01cm);
\end{pgfonlayer}{foreground layer}
\begin{pgfonlayer}{background layer}
\clip (0,0) rectangle (1,1); 
\draw [color=moyen]  (x1)--(x15);
\draw [color=moyen]  (x1)--(x16);
\draw [color=moyen]  (x2)--(x13);
\draw [color=moyen]  (x2)--(x15);
\draw [color=moyen]  (x2)--(x27);
\draw [color=moyen]  (x2)--(x29);
\draw [color=moyen]  (x2)--(x38);
\draw [color=moyen]  (x2)--(x40);
\draw [color=moyen]  (x2)--(x43);
\draw [color=moyen]  (x3)--(x7);
\draw [color=moyen]  (x3)--(x18);
\draw [color=moyen]  (x3)--(x24);
\draw [color=moyen]  (x3)--(x25);
\draw [color=moyen]  (x4)--(x14);
\draw [color=moyen]  (x4)--(x20);
\draw [color=moyen]  (x4)--(x34);
\draw [color=moyen]  (x5)--(x6);
\draw [color=moyen]  (x5)--(x21);
\draw [color=moyen]  (x5)--(x22);
\draw [color=moyen]  (x5)--(x26);
\draw [color=moyen]  (x5)--(x41);
\draw [color=moyen]  (x6)--(x22);
\draw [color=moyen]  (x6)--(x26);
\draw [color=moyen]  (x6)--(x28);
\draw [color=moyen]  (x6)--(x41);
\draw [color=moyen]  (x7)--(x24);
\draw [color=moyen]  (x8)--(x23);
\draw [color=moyen]  (x8)--(x33);
\draw [color=moyen]  (x8)--(x35);
\draw [color=moyen]  (x8)--(x39);
\draw [color=moyen]  (x9)--(x17);
\draw [color=moyen]  (x9)--(x21);
\draw [color=moyen]  (x9)--(x28);
\draw [color=moyen]  (x9)--(x32);
\draw [color=moyen]  (x10)--(x12);
\draw [color=moyen]  (x10)--(x23);
\draw [color=moyen]  (x10)--(x31);
\draw [color=moyen]  (x10)--(x37);
\draw [color=moyen]  (x11)--(x36);
\draw [color=moyen]  (x12)--(x23);
\draw [color=moyen]  (x12)--(x39);
\draw [color=moyen]  (x13)--(x15);
\draw [color=moyen]  (x13)--(x27);
\draw [color=moyen]  (x13)--(x29);
\draw [color=moyen]  (x13)--(x38);
\draw [color=moyen]  (x13)--(x40);
\draw [color=moyen]  (x13)--(x43);
\draw [color=moyen]  (x14)--(x24);
\draw [color=moyen]  (x14)--(x34);
\draw [color=moyen]  (x15)--(x16);
\draw [color=moyen]  (x15)--(x29);
\draw [color=moyen]  (x15)--(x38);
\draw [color=moyen]  (x15)--(x43);
\draw [color=moyen]  (x17)--(x18);
\draw [color=moyen]  (x17)--(x24);
\draw [color=moyen]  (x17)--(x32);
\draw [color=moyen]  (x17)--(x34);
\draw [color=moyen]  (x18)--(x24);
\draw [color=moyen]  (x18)--(x34);
\draw [color=moyen]  (x19)--(x30);
\draw [color=moyen]  (x19)--(x42);
\draw [color=moyen]  (x21)--(x28);
\draw [color=moyen]  (x21)--(x32);
\draw [color=moyen]  (x22)--(x26);
\draw [color=moyen]  (x22)--(x41);
\draw [color=moyen]  (x23)--(x39);
\draw [color=moyen]  (x24)--(x34);
\draw [color=moyen]  (x26)--(x41);
\draw [color=moyen]  (x27)--(x29);
\draw [color=moyen]  (x27)--(x38);
\draw [color=moyen]  (x27)--(x39);
\draw [color=moyen]  (x27)--(x40);
\draw [color=moyen]  (x27)--(x43);
\draw [color=moyen]  (x28)--(x32);
\draw [color=moyen]  (x29)--(x38);
\draw [color=moyen]  (x29)--(x40);
\draw [color=moyen]  (x29)--(x43);
\draw [color=moyen]  (x30)--(x37);
\draw [color=moyen]  (x30)--(x42);
\draw [color=moyen]  (x31)--(x37);
\draw [color=moyen]  (x33)--(x35);
\draw [color=moyen]  (x37)--(x42);
\draw [color=moyen]  (x38)--(x39);
\draw [color=moyen]  (x38)--(x43);
\draw [color=moyen]  (x40)--(x43);
\end{pgfonlayer}{background layer}
\end{tikzpicture}
\caption{A wireless network and its neighbors graph representation.}
\label{fig_neighbor} 
\end{figure}

Indeed, graphs can be generalized to more generic combinatorial
objects known as simplicial complexes. While graphs model binary
relations, simplicial complexes can represent higher order
relations. A simplicial complex is thus a combinatorial object made up
of vertices, edges, triangles, tetrahedra, and their $n$-dimensional
counterparts. 

Given a set of vertices $X$ and an integer $k$, a $k$-simplex is an
unordered subset of $k+1$ vertices $\{x_0,xÒ_1,\dots, x_k\}$ where $x_i\in
X, \forall i \in \{0,\dots,k\}$ and $x_i\not=x_j$ for all
$i\not=j$. Thus, a $0$-simplex is a vertex, a $1$-simplex an edge, a
$2$-simplex a triangle, a $3$-simplex a tetrahedron, etc. See
Fig.~\ref{fig_simplices} for instance. 

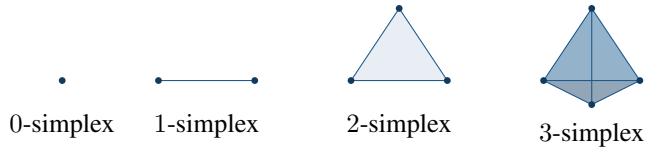
\begin{figure}[h]
  \centering
    \begin{tikzpicture}[scale=0.64]
\fill [color=fonce] (0,0) circle (2pt);
\node [below] at (0,-0.5) {$0$-simplex};
\draw [color=moyen] (2,0)--(4,0);
\fill [color=fonce] (2,0) circle (2pt);
\fill [color=fonce] (4,0) circle (2pt);
\node [below] at (3,-0.5) {$1$-simplex};
\fill [color=clair] (6,0)--(8,0)--(7,1.5);
\draw [color=moyen] (6,0)--(8,0)--(7,1.5)--(6,0);
\fill [color=fonce] (6,0) circle (2pt);
\fill [color=fonce] (8,0) circle (2pt);
\fill [color=fonce] (7,1.5) circle (2pt);
\node [below] at (7,-0.5) {$2$-simplex};
\fill [color=moyen, opacity=1] (10,0)--(12,0)--(11,1.5);
\fill [color=fonce,opacity=1] (10,0)--(12,0)--(11,-0.5);
\fill [color=clair,opacity=0.6] (10,0)--(11,-0.5)--(11,1.5);
\fill [color=clair,opacity=0.6] (11,-0.5)--(12,0)--(11,1.5);
\draw [color=moyen] (10,0)--(12,0)--(11,1.5)--(10,0);
\draw [color=moyen] (10,0)--(11,-0.5)--(11,1.5);
\draw [color=moyen] (12,0)--(11,-0.5);
\fill [color=fonce] (10,0) circle (2pt);
\fill [color=fonce] (12,0) circle (2pt);
\fill [color=fonce] (11,1.5) circle (2pt);
\fill [color=fonce] (11,-0.5) circle (2pt);
\node [below] at (11,-0.7) {$3$-simplex};
    \end{tikzpicture}
  \caption{Examples of $k$-simplices}\label{fig_simplices}
\end{figure}

Any subset of vertices included in the set of the $k+1$ vertices of a
$k$-simplex is a face of this $k$-simplex. A $k$-face is then a face
that is a $k$-simplex. Thus, a $k$-simplex has
exactly $k+1$ $(k-1)$-faces, which are $(k-1)$-simplices. For example,
a tetrahedron has four $3$-faces which are triangles. The inverse
notion of face is coface: if a simplex $S_1$ is a face of a larger
simplex $S_2$, then $S_2$ is a coface of $S_1$. As for faces, a
$k$-coface is a coface that is a $k$-simplex.

A simplicial complex is a collection of simplices which is closed with respect to
the inclusion of faces, i.e.\ all faces of a simplex are in the set of
simplices, and whenever two simplices intersect, they do so on a
common face. An abstract simplicial complex is a purely combinatorial
description of the geometric simplicial complex and therefore does not
need the property of intersection of faces. In this article, we are
only interested in the combinatorial description of a simplicial
complex, that is why we will only consider abstract simplicial
complexes, even if the adjective ``abstract'' may sometimes be
dropped. Let us denote by $x_0,x_1,\dots$ some vertices, and then
write $[x_0,\dots,x_k]$ a $k$-simplex for any $k$ integer. An example
of  an abstract simplicial complex with five $0$-simplices
$x_0,\dots, x_4$, six $1$-simplices
$[x_0,x_1],[x_0,x_2],[x_1,x_2],[x_1,x_4],[x_2,x_3],[x_3,x_4]$, and one $2$-simplex $[x_0,x_1,x_2]$
can be seen in Fig.~\ref{fig_complex}.

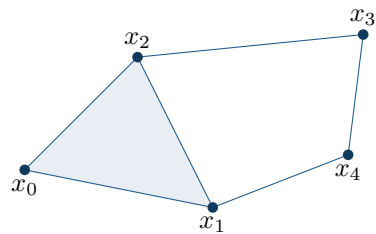
\begin{figure}[h]
  \centering
  \begin{tikzpicture}
      \fill [color=clair] (-1.5,0)--(1,-0.5)--(0,1.5); 
\draw [color=moyen] (-1.5,0)--(1,-0.5)--(0,1.5)--(-1.5,0); 
\draw [color=moyen] (1,-0.5)--(2.8,0.2)--(3,1.8)--(0,1.5); 
\node [below] at (-1.5,0) {$x_0$}; 
\node [below] at (1,-0.5) {$x_1$}; 
\node [above] at (0,1.5) {$x_2$}; 
\node [above] at (3,1.8) {$x_3$}; 
\node [below] at (2.8,0.2) {$x_4$};
\fill [color=fonce] (-1.5,0) circle (2pt);
\fill [color=fonce] (1,-0.5) circle (2pt);
\fill [color=fonce] (0,1.5) circle (2pt);
\fill [color=fonce] (3,1.8) circle (2pt);
\fill [color=fonce] (2.8,0.2) circle (2pt);
    \end{tikzpicture}
  \caption{Example of an abstract simplicial complex}\label{fig_complex}
\end{figure}

The abstract simplicial complex that can exactly represent the
topology of a wireless network (see the Nerve lemma in
\cite{ghrist_coverage_2005}) is the \u{C}ech complex whose definition
is: 
\begin{definition}[\u{C}ech complex]
Let $(X,d)$ be a metric space, $\omega$ a finite set of points in
  $X$, and $r$ a real positive number.
The \u{C}ech complex of parameter $r$ on the set of vertices $\omega$, denoted
$\mathcal{C}_{r}(\omega)$, is the abstract simplicial complex whose
$k$-simplices are the unordered $(k+1)$-tuples of vertices in $\omega$
for which the intersection of the $k+1$ balls of radius $r$ centered
at the $k+1$ vertices is non empty.
\end{definition}

However, the \u{C}ech complex can be difficult to build since one must
know the intersection of every three balls. Therefore, we are
interested in the approximation of the \u{C}ech complex, the
Vietoris-Rips complex:
\begin{definition}[Vietoris-Rips complex]
 Let $(X,d)$ be a metric space, $\omega$ a finite set of points in
  $X$, and $r$ a real positive number. The Vietoris-Rips
  complex of parameter $r$ of $\omega$, denoted
  $\mathcal{R}_{r}(\omega)$, is the abstract simplicial complex
  whose $k$-simplices correspond to the unordered $(k+1)$-tuples of
  vertices in $\omega$ which are pairwise within distance less than
  $2r$ of each other.
\end{definition}

The Vietoris-Rips complex is easier to build than the \u{C}ech complex
since it is build only based on the neighbors graph information, and
it provides a good approximation for the network's topology if the
network nodes are deployed according to a Poisson point process
\cite{yan_accuracy_2012}. An example of a Vietoris-Rips complex can be
seen in Fig.~\ref{fig_rips}.

\begin{figure}[h]
  \centering
\begin{tikzpicture}[scale=4]
\draw  [draw=moyen] (0,0) rectangle (1,1); 
\useasboundingbox (0,0) rectangle (1,1);
\coordinate (x1) at (0.9238,0.0566);
\coordinate (x2) at (0.7228,0.3340);
\coordinate (x3) at (0.2711,0.6092);
\coordinate (x4) at (0.0474,0.7752);
\coordinate (x5) at (0.2200,0.1053);
\coordinate (x6) at (0.4072,0.1743);
\coordinate (x7) at (0.2512,0.6968);
\coordinate (x8) at (0.5722,0.7114);
\coordinate (x9) at (0.1261,0.2915);
\coordinate (x10) at (0.8779,0.5168);
\coordinate (x11) at (0.2832,0.9846);
\coordinate (x12) at (0.7325,0.5803);
\coordinate (x13) at (0.7567,0.2407);
\coordinate (x14) at (0.0475,0.7101);
\coordinate (x15) at (0.8411,0.2162);
\coordinate (x16) at (0.9994,0.2155);
\coordinate (x17) at (0.1194,0.4624);
\coordinate (x18) at (0.1374,0.5071);
\coordinate (x19) at (0.8713,0.9582);
\coordinate (x20) at (0.0148,0.9434);
\coordinate (x21) at (0.2313,0.2756);
\coordinate (x22) at (0.3645,0.0602);
\coordinate (x23) at (0.7080,0.5890);
\coordinate (x24) at (0.1893,0.5744);
\coordinate (x25) at (0.4145,0.5612);
\coordinate (x26) at (0.3700,0.1255);
\coordinate (x27) at (0.6392,0.3585);
\coordinate (x28) at (0.2951,0.3264);
\coordinate (x29) at (0.7366,0.3309);
\coordinate (x30) at (0.8620,0.7739);
\coordinate (x31) at (0.9619,0.5384);
\coordinate (x32) at (0.1431,0.3001);
\coordinate (x33) at (0.5956,0.8182);
\coordinate (x34) at (0.0261,0.5891);
\coordinate (x35) at (0.6296,0.7746);
\coordinate (x36) at (0.4035,0.9296);
\coordinate (x37) at (0.8850,0.7112);
\coordinate (x38) at (0.7402,0.3712);
\coordinate (x39) at (0.6438,0.5415);
\coordinate (x40) at (0.6236,0.1856);
\coordinate (x41) at (0.2666,0.0592);
\coordinate (x42) at (0.8421,0.8227);
\coordinate (x43) at (0.7540,0.2975);
\begin{pgfonlayer}{foreground layer}
\clip (0,0) rectangle (1,1); 
\draw [fill=fonce] (x1) circle (0.01cm);
\draw [fill=fonce] (x2) circle (0.01cm);
\draw [fill=fonce] (x3) circle (0.01cm);
\draw [fill=fonce] (x4) circle (0.01cm);
\draw [fill=fonce] (x6) circle (0.01cm);
\draw [fill=fonce] (x9) circle (0.01cm);
\draw [fill=fonce] (x10) circle (0.01cm);
\draw [fill=fonce] (x11) circle (0.01cm);
\draw [fill=fonce] (x14) circle (0.01cm);
\draw [fill=fonce] (x15) circle (0.01cm);
\draw [fill=fonce] (x16) circle (0.01cm);
\draw [fill=fonce] (x17) circle (0.01cm);
\draw [fill=fonce] (x18) circle (0.01cm);
\draw [fill=fonce] (x20) circle (0.01cm);
\draw [fill=fonce] (x21) circle (0.01cm);
\draw [fill=fonce] (x22) circle (0.01cm);
\draw [fill=fonce] (x23) circle (0.01cm);
\draw [fill=fonce] (x24) circle (0.01cm);
\draw [fill=fonce] (x25) circle (0.01cm);
\draw [fill=fonce] (x27) circle (0.01cm);
\draw [fill=fonce] (x28) circle (0.01cm);
\draw [fill=fonce] (x30) circle (0.01cm);
\draw [fill=fonce] (x33) circle (0.01cm);
\draw [fill=fonce] (x35) circle (0.01cm);
\draw [fill=fonce] (x36) circle (0.01cm);
\draw [fill=fonce] (x37) circle (0.01cm);
\draw [fill=fonce] (x38) circle (0.01cm);
\draw [fill=fonce] (x39) circle (0.01cm);
\draw [fill=fonce] (x40) circle (0.01cm);
\draw [fill=fonce] (x41) circle (0.01cm);
\draw [fill=fonce] (x42) circle (0.01cm);
\draw [fill=fonce] (x5) circle (0.01cm);
\draw [fill=fonce] (x5) circle (0.01cm);
\draw [fill=fonce] (x19) circle (0.01cm);
\draw [fill=fonce] (x29) circle (0.01cm);
\draw [fill=fonce] (x13) circle (0.01cm);
\draw [fill=fonce] (x7) circle (0.01cm);
\draw [fill=fonce] (x8) circle (0.01cm);
\draw [fill=fonce] (x31) circle (0.01cm);
\draw [fill=fonce] (x32) circle (0.01cm);
\draw [fill=fonce] (x26) circle (0.01cm);
\draw [fill=fonce] (x34) circle (0.01cm);
\draw [fill=fonce] (x12) circle (0.01cm);
\end{pgfonlayer}{foreground layer}
\begin{pgfonlayer}{background layer}
\clip (0,0) rectangle (1,1); 
\draw [fill,color=clair,opacity=0.5] (x1) circle (0.1cm);
\draw [fill,color=clair,opacity=0.5] (x2) circle (0.1cm);
\draw [fill,color=clair,opacity=0.5] (x3) circle (0.1cm);
\draw [fill,color=clair,opacity=0.5] (x4) circle (0.1cm);
\draw [fill,color=clair,opacity=0.5] (x5) circle (0.1cm);
  \draw [fill,color=clair,opacity=0.5] (x6) circle (  0.1cm);
   \draw [fill,color=clair,opacity=0.5] (x7) circle ( 0.1cm);
   \draw [fill,color=clair,opacity=0.5] (x8) circle ( 0.1cm);
   \draw [fill,color=clair,opacity=0.5] (x9) circle ( 0.1cm);
   \draw [fill,color=clair,opacity=0.5] (x10) circle ( 0.1cm);
   \draw [fill,color=clair,opacity=0.5] (x11) circle ( 0.1cm);
    \draw [fill,color=clair,opacity=0.5] (x12) circle (0.1cm);
   \draw [fill,color=clair,opacity=0.5] (x13) circle ( 0.1cm);
  \draw [fill,color=clair,opacity=0.5] (x14) circle (  0.1cm);
   \draw [fill,color=clair,opacity=0.5] (x15) circle ( 0.1cm);
   \draw [fill,color=clair,opacity=0.5] (x16) circle ( 0.1cm);
   \draw [fill,color=clair,opacity=0.5] (x17) circle ( 0.1cm);
   \draw [fill,color=clair,opacity=0.5] (x18) circle ( 0.1cm);
   \draw [fill,color=clair,opacity=0.5] (x19) circle ( 0.1cm);
   \draw [fill,color=clair,opacity=0.5] (x20) circle ( 0.1cm);
   \draw [fill,color=clair,opacity=0.5] (x21) circle ( 0.1cm);
   \draw [fill,color=clair,opacity=0.5] (x22) circle ( 0.1cm);
   \draw [fill,color=clair,opacity=0.5] (x23) circle ( 0.1cm);
   \draw [fill,color=clair,opacity=0.5] (x24) circle (  0.1cm);
   \draw [fill,color=clair,opacity=0.5] (x25) circle (  0.1cm);
   \draw [fill,color=clair,opacity=0.5] (x26) circle (  0.1cm);
   \draw [fill,color=clair,opacity=0.5] (x27) circle ( 0.1cm);
   \draw [fill,color=clair,opacity=0.5] (x28) circle (  0.1cm);
\draw [fill,color=clair,opacity=0.5] (x29) circle ( 0.1cm);
\draw [fill,color=clair,opacity=0.5] (x30) circle ( 0.1cm);
\draw [fill,color=clair,opacity=0.5] (x31) circle ( 0.1cm);
\draw [fill,color=clair,opacity=0.5] (x32) circle ( 0.1cm);
 \draw [fill,color=clair,opacity=0.5] (x33) circle ( 0.1cm);
 \draw [fill,color=clair,opacity=0.5] (x34) circle (  0.1cm);
\draw [fill,color=clair,opacity=0.5] (x35) circle ( 0.1cm);
\draw [fill,color=clair,opacity=0.5] (x36) circle ( 0.1cm);
\draw [fill,color=clair,opacity=0.5] (x37) circle ( 0.1cm);
\draw [fill,color=clair,opacity=0.5] (x38) circle ( 0.1cm);
\draw [fill,color=clair,opacity=0.5] (x39) circle ( 0.1cm);
\draw [fill,color=clair,opacity=0.5] (x40) circle ( 0.1cm);
\draw [fill,color=clair,opacity=0.5] (x41) circle ( 0.1cm);
\draw [fill,color=clair,opacity=0.5] (x42) circle ( 0.1cm);
\draw [color=moyen] (x1) circle (0.1cm);
\draw [color=moyen]  (x2) circle (0.1cm);
\draw [color=moyen]  (x3) circle (0.1cm);
\draw [color=moyen] (x4) circle (0.1cm);
\draw [color=moyen] (x5) circle (0.1cm);
\draw [color=moyen] (x6) circle (  0.1cm);
\draw [color=moyen]  (x7) circle ( 0.1cm);
\draw [color=moyen] (x8) circle ( 0.1cm);
\draw [color=moyen] (x9) circle ( 0.1cm);
\draw [color=moyen]  (x10) circle ( 0.1cm);
\draw [color=moyen]  (x11) circle ( 0.1cm);
\draw [color=moyen] (x12) circle (0.1cm);
\draw [color=moyen] (x13) circle ( 0.1cm);
\draw [color=moyen] (x14) circle (  0.1cm);
\draw [color=moyen]  (x15) circle ( 0.1cm);
\draw [color=moyen]  (x16) circle ( 0.1cm);
\draw [color=moyen]  (x17) circle ( 0.1cm);
\draw [color=moyen]  (x18) circle ( 0.1cm);
\draw [color=moyen]  (x19) circle ( 0.1cm);
\draw [color=moyen]  (x20) circle ( 0.1cm);
\draw [color=moyen] (x21) circle ( 0.1cm);
\draw [color=moyen] (x22) circle ( 0.1cm);
\draw [color=moyen]  (x23) circle ( 0.1cm);
\draw [color=moyen]  (x24) circle (  0.1cm);
\draw [color=moyen]  (x25) circle (  0.1cm);
\draw [color=moyen] (x26) circle (  0.1cm);
\draw [color=moyen]  (x27) circle ( 0.1cm);
\draw [color=moyen]  (x28) circle (  0.1cm);
\draw [color=moyen] (x29) circle ( 0.1cm);
\draw [color=moyen] (x30) circle ( 0.1cm);
\draw [color=moyen] (x31) circle ( 0.1cm);
\draw [color=moyen] (x32) circle ( 0.1cm);
\draw [color=moyen] (x33) circle ( 0.1cm);
\draw [color=moyen] (x34) circle (  0.1cm);
\draw [color=moyen]  (x35) circle ( 0.1cm);
\draw [color=moyen]  (x36) circle ( 0.1cm);
\draw [color=moyen] (x37) circle ( 0.1cm);
\draw [color=moyen]  (x38) circle ( 0.1cm);
\draw [color=moyen]  (x39) circle ( 0.1cm);
\draw [color=moyen]  (x40) circle ( 0.1cm);
\draw [color=moyen]  (x41) circle ( 0.1cm);
\draw [color=moyen]  (x42) circle ( 0.1cm);
\end{pgfonlayer}{background layer}
\end{tikzpicture} 
\begin{tikzpicture}[scale=4]
\draw  [draw=moyen] (0,0) rectangle (1,1); 
\useasboundingbox (0,0) rectangle (1,1);
\coordinate (x1) at (0.9238,0.0566);
\coordinate (x2) at (0.7228,0.3340);
\coordinate (x3) at (0.2711,0.6092);
\coordinate (x4) at (0.0474,0.7752);
\coordinate (x5) at (0.2200,0.1053);
\coordinate (x6) at (0.4072,0.1743);
\coordinate (x7) at (0.2512,0.6968);
\coordinate (x8) at (0.5722,0.7114);
\coordinate (x9) at (0.1261,0.2915);
\coordinate (x10) at (0.8779,0.5168);
\coordinate (x11) at (0.2832,0.9846);
\coordinate (x12) at (0.7325,0.5803);
\coordinate (x13) at (0.7567,0.2407);
\coordinate (x14) at (0.0475,0.7101);
\coordinate (x15) at (0.8411,0.2162);
\coordinate (x16) at (0.9994,0.2155);
\coordinate (x17) at (0.1194,0.4624);
\coordinate (x18) at (0.1374,0.5071);
\coordinate (x19) at (0.8713,0.9582);
\coordinate (x20) at (0.0148,0.9434);
\coordinate (x21) at (0.2313,0.2756);
\coordinate (x22) at (0.3645,0.0602);
\coordinate (x23) at (0.7080,0.5890);
\coordinate (x24) at (0.1893,0.5744);
\coordinate (x25) at (0.4145,0.5612);
\coordinate (x26) at (0.3700,0.1255);
\coordinate (x27) at (0.6392,0.3585);
\coordinate (x28) at (0.2951,0.3264);
\coordinate (x29) at (0.7366,0.3309);
\coordinate (x30) at (0.8620,0.7739);
\coordinate (x31) at (0.9619,0.5384);
\coordinate (x32) at (0.1431,0.3001);
\coordinate (x33) at (0.5956,0.8182);
\coordinate (x34) at (0.0261,0.5891);
\coordinate (x35) at (0.6296,0.7746);
\coordinate (x36) at (0.4035,0.9296);
\coordinate (x37) at (0.8850,0.7112);
\coordinate (x38) at (0.7402,0.3712);
\coordinate (x39) at (0.6438,0.5415);
\coordinate (x40) at (0.6236,0.1856);
\coordinate (x41) at (0.2666,0.0592);
\coordinate (x42) at (0.8421,0.8227);
\coordinate (x43) at (0.7540,0.2975);
\begin{pgfonlayer}{foreground layer}
\clip (0,0) rectangle (1,1); 
\draw [fill=fonce] (x1) circle (0.01cm);
\draw [fill=fonce] (x2) circle (0.01cm);
\draw [fill=fonce] (x3) circle (0.01cm);
\draw [fill=fonce] (x4) circle (0.01cm);
\draw [fill=fonce] (x6) circle (0.01cm);
\draw [fill=fonce] (x9) circle (0.01cm);
\draw [fill=fonce] (x10) circle (0.01cm);
\draw [fill=fonce] (x11) circle (0.01cm);
\draw [fill=fonce] (x14) circle (0.01cm);
\draw [fill=fonce] (x15) circle (0.01cm);
\draw [fill=fonce] (x16) circle (0.01cm);
\draw [fill=fonce] (x17) circle (0.01cm);
\draw [fill=fonce] (x18) circle (0.01cm);
\draw [fill=fonce] (x20) circle (0.01cm);
\draw [fill=fonce] (x21) circle (0.01cm);
\draw [fill=fonce] (x22) circle (0.01cm);
\draw [fill=fonce] (x23) circle (0.01cm);
\draw [fill=fonce] (x24) circle (0.01cm);
\draw [fill=fonce] (x25) circle (0.01cm);
\draw [fill=fonce] (x27) circle (0.01cm);
\draw [fill=fonce] (x28) circle (0.01cm);
\draw [fill=fonce] (x30) circle (0.01cm);
\draw [fill=fonce] (x33) circle (0.01cm);
\draw [fill=fonce] (x35) circle (0.01cm);
\draw [fill=fonce] (x36) circle (0.01cm);
\draw [fill=fonce] (x37) circle (0.01cm);
\draw [fill=fonce] (x38) circle (0.01cm);
\draw [fill=fonce] (x39) circle (0.01cm);
\draw [fill=fonce] (x40) circle (0.01cm);
\draw [fill=fonce] (x41) circle (0.01cm);
\draw [fill=fonce] (x42) circle (0.01cm);
\draw [fill=fonce] (x5) circle (0.01cm);
\draw [fill=fonce] (x5) circle (0.01cm);
\draw [fill=fonce] (x19) circle (0.01cm);
\draw [fill=fonce] (x29) circle (0.01cm);
\draw [fill=fonce] (x13) circle (0.01cm);
\draw [fill=fonce] (x7) circle (0.01cm);
\draw [fill=fonce] (x8) circle (0.01cm);
\draw [fill=fonce] (x31) circle (0.01cm);
\draw [fill=fonce] (x32) circle (0.01cm);
\draw [fill=fonce] (x26) circle (0.01cm);
\draw [fill=fonce] (x34) circle (0.01cm);
\draw [fill=fonce] (x12) circle (0.01cm);
\end{pgfonlayer}{foreground layer}
\begin{pgfonlayer}{background layer}
\clip (0,0) rectangle (1,1); 
\draw [fill,color=clair,opacity=0.5] (x1)--(x15)--(x16);
\draw [fill,color=clair,opacity=0.5] (x2)--(x13)--(x15);
\draw [fill,color=clair,opacity=0.5] (x2)--(x13)--(x27);
\draw [fill,color=clair,opacity=0.5] (x2)--(x13)--(x29);
\draw [fill,color=clair,opacity=0.5] (x2)--(x13)--(x38);
\draw [fill,color=clair,opacity=0.5] (x2)--(x13)--(x40);
\draw [fill,color=clair,opacity=0.5] (x2)--(x13)--(x43);
\draw [fill,color=clair,opacity=0.5] (x2)--(x15)--(x29);
\draw [fill,color=clair,opacity=0.5] (x2)--(x15)--(x38);
\draw [fill,color=clair,opacity=0.5] (x2)--(x15)--(x43);
\draw [fill,color=clair,opacity=0.5] (x2)--(x27)--(x29);
\draw [fill,color=clair,opacity=0.5] (x2)--(x27)--(x38);
\draw [fill,color=clair,opacity=0.5] (x2)--(x27)--(x40);
\draw [fill,color=clair,opacity=0.5] (x2)--(x27)--(x43);
\draw [fill,color=clair,opacity=0.5] (x2)--(x29)--(x38);
\draw [fill,color=clair,opacity=0.5] (x2)--(x29)--(x40);
\draw [fill,color=clair,opacity=0.5] (x2)--(x29)--(x43);
\draw [fill,color=clair,opacity=0.5] (x2)--(x38)--(x43);
\draw [fill,color=clair,opacity=0.5] (x2)--(x40)--(x43);
\draw [fill,color=clair,opacity=0.5] (x3)--(x7)--(x24);
\draw [fill,color=clair,opacity=0.5] (x3)--(x18)--(x24);
\draw [fill,color=clair,opacity=0.5] (x4)--(x14)--(x34);
\draw [fill,color=clair,opacity=0.5] (x5)--(x6)--(x22);
\draw [fill,color=clair,opacity=0.5] (x5)--(x6)--(x26);
\draw [fill,color=clair,opacity=0.5] (x5)--(x6)--(x41);
\draw [fill,color=clair,opacity=0.5] (x5)--(x22)--(x26);
\draw [fill,color=clair,opacity=0.5] (x5)--(x22)--(x41);
\draw [fill,color=clair,opacity=0.5] (x5)--(x26)--(x41);
\draw [fill,color=clair,opacity=0.5] (x6)--(x22)--(x26);
\draw [fill,color=clair,opacity=0.5] (x6)--(x22)--(x41);
\draw [fill,color=clair,opacity=0.5] (x6)--(x26)--(x41);
\draw [fill,color=clair,opacity=0.5] (x8)--(x23)--(x39);
\draw [fill,color=clair,opacity=0.5] (x8)--(x33)--(x35);
\draw [fill,color=clair,opacity=0.5] (x9)--(x17)--(x32);
\draw [fill,color=clair,opacity=0.5] (x9)--(x21)--(x28);
\draw [fill,color=clair,opacity=0.5] (x9)--(x21)--(x32);
\draw [fill,color=clair,opacity=0.5] (x9)--(x28)--(x32);
\draw [fill,color=clair,opacity=0.5] (x10)--(x12)--(x23);
\draw [fill,color=clair,opacity=0.5] (x10)--(x31)--(x37);
\draw [fill,color=clair,opacity=0.5] (x12)--(x23)--(x39);
\draw [fill,color=clair,opacity=0.5] (x13)--(x15)--(x29);
\draw [fill,color=clair,opacity=0.5] (x13)--(x15)--(x38);
\draw [fill,color=clair,opacity=0.5] (x13)--(x15)--(x43);
\draw [fill,color=clair,opacity=0.5] (x13)--(x27)--(x29);
\draw [fill,color=clair,opacity=0.5] (x13)--(x27)--(x38);
\draw [fill,color=clair,opacity=0.5] (x13)--(x27)--(x40);
\draw [fill,color=clair,opacity=0.5] (x13)--(x27)--(x43);
\draw [fill,color=clair,opacity=0.5] (x13)--(x29)--(x38);
\draw [fill,color=clair,opacity=0.5] (x13)--(x29)--(x40);
\draw [fill,color=clair,opacity=0.5] (x13)--(x29)--(x43);
\draw [fill,color=clair,opacity=0.5] (x13)--(x38)--(x43);
\draw [fill,color=clair,opacity=0.5] (x13)--(x40)--(x43);
\draw [fill,color=clair,opacity=0.5] (x14)--(x24)--(x34);
\draw [fill,color=clair,opacity=0.5] (x15)--(x29)--(x38);
\draw [fill,color=clair,opacity=0.5] (x15)--(x29)--(x43);
\draw [fill,color=clair,opacity=0.5] (x15)--(x38)--(x43);
\draw [fill,color=clair,opacity=0.5] (x17)--(x18)--(x24);
\draw [fill,color=clair,opacity=0.5] (x17)--(x18)--(x34);
\draw [fill,color=clair,opacity=0.5] (x17)--(x24)--(x34);
\draw [fill,color=clair,opacity=0.5] (x18)--(x24)--(x34);
\draw [fill,color=clair,opacity=0.5] (x19)--(x30)--(x42);
\draw [fill,color=clair,opacity=0.5] (x21)--(x28)--(x32);
\draw [fill,color=clair,opacity=0.5] (x22)--(x26)--(x41);
\draw [fill,color=clair,opacity=0.5] (x27)--(x29)--(x38);
\draw [fill,color=clair,opacity=0.5] (x27)--(x29)--(x40);
\draw [fill,color=clair,opacity=0.5] (x27)--(x29)--(x43);
\draw [fill,color=clair,opacity=0.5] (x27)--(x38)--(x39);
\draw [fill,color=clair,opacity=0.5] (x27)--(x38)--(x43);
\draw [fill,color=clair,opacity=0.5] (x27)--(x40)--(x43);
\draw [fill,color=clair,opacity=0.5] (x29)--(x38)--(x43);
\draw [fill,color=clair,opacity=0.5] (x29)--(x40)--(x43);
\draw [fill,color=clair,opacity=0.5] (x30)--(x37)--(x42);
\draw [color=moyen]  (x1)--(x15);
\draw [color=moyen]  (x1)--(x16);
\draw [color=moyen]  (x2)--(x13);
\draw [color=moyen]  (x2)--(x15);
\draw [color=moyen]  (x2)--(x27);
\draw [color=moyen]  (x2)--(x29);
\draw [color=moyen]  (x2)--(x38);
\draw [color=moyen]  (x2)--(x40);
\draw [color=moyen]  (x2)--(x43);
\draw [color=moyen]  (x3)--(x7);
\draw [color=moyen]  (x3)--(x18);
\draw [color=moyen]  (x3)--(x24);
\draw [color=moyen]  (x3)--(x25);
\draw [color=moyen]  (x4)--(x14);
\draw [color=moyen]  (x4)--(x20);
\draw [color=moyen]  (x4)--(x34);
\draw [color=moyen]  (x5)--(x6);
\draw [color=moyen]  (x5)--(x21);
\draw [color=moyen]  (x5)--(x22);
\draw [color=moyen]  (x5)--(x26);
\draw [color=moyen]  (x5)--(x41);
\draw [color=moyen]  (x6)--(x22);
\draw [color=moyen]  (x6)--(x26);
\draw [color=moyen]  (x6)--(x28);
\draw [color=moyen]  (x6)--(x41);
\draw [color=moyen]  (x7)--(x24);
\draw [color=moyen]  (x8)--(x23);
\draw [color=moyen]  (x8)--(x33);
\draw [color=moyen]  (x8)--(x35);
\draw [color=moyen]  (x8)--(x39);
\draw [color=moyen]  (x9)--(x17);
\draw [color=moyen]  (x9)--(x21);
\draw [color=moyen]  (x9)--(x28);
\draw [color=moyen]  (x9)--(x32);
\draw [color=moyen]  (x10)--(x12);
\draw [color=moyen]  (x10)--(x23);
\draw [color=moyen]  (x10)--(x31);
\draw [color=moyen]  (x10)--(x37);
\draw [color=moyen]  (x11)--(x36);
\draw [color=moyen]  (x12)--(x23);
\draw [color=moyen]  (x12)--(x39);
\draw [color=moyen]  (x13)--(x15);
\draw [color=moyen]  (x13)--(x27);
\draw [color=moyen]  (x13)--(x29);
\draw [color=moyen]  (x13)--(x38);
\draw [color=moyen]  (x13)--(x40);
\draw [color=moyen]  (x13)--(x43);
\draw [color=moyen]  (x14)--(x24);
\draw [color=moyen]  (x14)--(x34);
\draw [color=moyen]  (x15)--(x16);
\draw [color=moyen]  (x15)--(x29);
\draw [color=moyen]  (x15)--(x38);
\draw [color=moyen]  (x15)--(x43);
\draw [color=moyen]  (x17)--(x18);
\draw [color=moyen]  (x17)--(x24);
\draw [color=moyen]  (x17)--(x32);
\draw [color=moyen]  (x17)--(x34);
\draw [color=moyen]  (x18)--(x24);
\draw [color=moyen]  (x18)--(x34);
\draw [color=moyen]  (x19)--(x30);
\draw [color=moyen]  (x19)--(x42);
\draw [color=moyen]  (x21)--(x28);
\draw [color=moyen]  (x21)--(x32);
\draw [color=moyen]  (x22)--(x26);
\draw [color=moyen]  (x22)--(x41);
\draw [color=moyen]  (x23)--(x39);
\draw [color=moyen]  (x24)--(x34);
\draw [color=moyen]  (x26)--(x41);
\draw [color=moyen]  (x27)--(x29);
\draw [color=moyen]  (x27)--(x38);
\draw [color=moyen]  (x27)--(x39);
\draw [color=moyen]  (x27)--(x40);
\draw [color=moyen]  (x27)--(x43);
\draw [color=moyen]  (x28)--(x32);
\draw [color=moyen]  (x29)--(x38);
\draw [color=moyen]  (x29)--(x40);
\draw [color=moyen]  (x29)--(x43);
\draw [color=moyen]  (x30)--(x37);
\draw [color=moyen]  (x30)--(x42);
\draw [color=moyen]  (x31)--(x37);
\draw [color=moyen]  (x33)--(x35);
\draw [color=moyen]  (x37)--(x42);
\draw [color=moyen]  (x38)--(x39);
\draw [color=moyen]  (x38)--(x43);
\draw [color=moyen]  (x40)--(x43);
\end{pgfonlayer}{background layer}
\end{tikzpicture}
\caption{A wireless network and its Vietoris-Rips complex.}
\label{fig_rips} 
\end{figure}
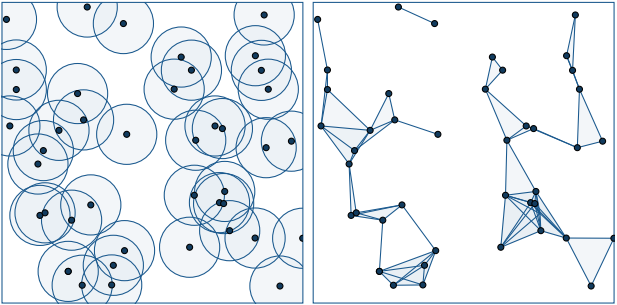

\subsection{Algebraic Topology}
Given an abstract simplicial complex, one can define an orientation on
the simplices by defining an order on the vertices, where a change in the
orientation corresponds to a change in the sign:
\begin{figure}[h]
 \begin{tikzpicture}[scale=0.7]
\coordinate (x1) at (0,0);
\coordinate (x2) at (2,0);
\coordinate (x3) at (0,-2);
\coordinate (x4) at (1,-1);
\coordinate (x5) at (2,-2);
 \fill [color=clair] (x3)--(x4)--(x5);
\draw[draw=moyen, arrows={-triangle 45}] (x1)--(x2);
\draw[draw=moyen](x3)--(x4);
\draw[draw=moyen](x4)--(x5);
\draw[draw=moyen] (x5)--(x3);
\draw[->][color=fonce] (1,-1.8) arc (270:-30:.2);
\draw [fill=fonce] (x1) circle (2pt);
\draw [fill=fonce] (x2) circle (2pt);
\draw [fill=fonce] (x3) circle (2pt);
\draw [fill=fonce] (x4) circle (2pt);
\draw [fill=fonce] (x5) circle (2pt);
\node [below] at (x1) {$x_0$};
\node [below] at (x2) {$x_1$};
\node [below] at (x3) {$x_0$};
\node [above] at (x4) {$x_1$};
\node [below] at (x5) {$x_2$};
\node [right] at (3,0) {$[x_0,x_1]=-[x_1,x_0]$};
\node [right] at (3,-2) {$[x_0,x_1,x_2]=-[x_0,x_2,x_1]$};
 \end{tikzpicture}
\end{figure}

Then a change of sign corresponds to a swap between two vertices:
\begin{align*}
  [x_0,\dots,x_i,\dots,x_j,\dots,&x_k]=\\
&-[x_0,\dots,x_j,\dots,x_i,\dots,x_k].
\end{align*}

Then let us define the vector spaces of the $k$-simplices of a
simplicial complex:
\begin{definition}
  Let $S$ be an abstract simplicial complex.

 For any integer $k$,
  $\mathscr{C}_k(S)$ is the vector space spanned by the set of oriented
  $k$-simplices of $S$.
\end{definition}

Then we can define a boundary map on these vector spaces:
\begin{definition}
Let $S$ be an abstract simplicial complex and $\mathscr{C}_k(S)$ the
vector space of its $k$-simplices for any $k$ integer.

The boundary map $\partial_k$ is defined as the linear
  transformation $\partial_k: \mathscr{C}_k(S)\rightarrow
  \mathscr{C}_{k-1}(S)$ which acts on the 
  basis elements $[x_0,\dots,x_k]$ of $\mathscr{C}_k(S)$ via:
  \begin{eqnarray*}
    \partial_k [x_0,\dots,x_k]= \sum_{i=0}^{k}{(-1)}^i [x_0,\dots,x_{i-1},x_{i+1},\dots,x_k].
  \end{eqnarray*}
\end{definition}

For example, for a $2$-simplex we have:
\begin{figure}[h]
\centering 
\begin{tikzpicture}[scale=0.7]
\coordinate (x1) at (0,0);
\coordinate (x2) at (1,1);
\coordinate (x3) at (2,0);
\coordinate (x4) at (4,0);
\coordinate (x5) at (5,1);
\coordinate (x6) at (6,0);
 \fill [color=clair] (x1)--(x2)--(x3);
\draw[draw=moyen] (x1)--(x2);
\draw[draw=moyen](x2)--(x3);
\draw[draw=moyen](x3)--(x1);
\draw[draw=moyen, arrows={-triangle 45}] (x4)->(x5);
\draw[draw=moyen, arrows={-triangle 45}] (x5)->(x6);
\draw[draw=moyen, arrows={-triangle 45}] (x6)->(x4);
\draw[->][color=fonce] (1,.2) arc (270:-30:.2);
\draw [fill=fonce] (x1) circle (2pt);
\draw [fill=fonce] (x2) circle (2pt);
\draw [fill=fonce] (x3) circle (2pt);
\draw [fill=fonce] (x4) circle (2pt);
\draw [fill=fonce] (x5) circle (2pt);
\draw [fill=fonce] (x6) circle (2pt);
\node [below] at (x1) {$x_0$};
\node [above] at (x2) {$x_1$};
\node [below] at (x3) {$x_2$};
\node [below] at (x4) {$x_0$};
\node [above] at (x5) {$x_1$};
\node [below] at (x6) {$x_2$};
\node [below] at (1,-0.5) {\small $\partial_2([x_0,x_1,x_2])$};
\node [below] at (3,-0.75) {\small $=$};
\node [below] at (6,-0.5) {\small $[x_1,x_2]-[x_0,x_2]+[x_0,x_1]$};
 \end{tikzpicture}
\end{figure}

The boundary map on any $k$-simplex, is the cycle of its $(k-1)$-faces.
This map gives rise to a chain complex (a sequence of vector spaces
and linear transformations):
\begin{eqnarray*}
  \ldots \!\stackrel{\partial_{k+2}}{\longrightarrow}C_{k+1}
  \stackrel{\partial_{k+1}}{\longrightarrow}C_{k}
  \stackrel{\partial_{k}}{\longrightarrow}C_{k-1}
  \stackrel{\partial_{k-1}}{\longrightarrow}
  \ldots\!
  \stackrel{\partial_1}{\longrightarrow}C_{0}
  \stackrel{\partial_0}{\longrightarrow}0. 
\end{eqnarray*}

We can see on our previous abstract simplicial complex example of
Fig.~\ref{fig_complex} the computation of the three first boundary
maps:
\begin{displaymath}
  \partial_0=
  \bordermatrix{
    ~ & [x_0]&[x_1]&[x_2]&[x_3]&[x_4]\cr
   & 0&0&0&0&0\cr
  }
\end{displaymath}
\begin{displaymath}
 \partial_1\!=\!\!
  \bordermatrix{
     ~ & [x_0x_1]\!\!\!&[x_0x_2]\!\!\!&[x_1x_2]\!\!\!&[x_1x_4]\!\!\!&[x_2x_3]\!\!\!&[x_3x_4]\!\!\! \cr
    [x_0]&-1&-1&0&0&0&0\cr
    [x_1]& 1&0&-1&-1&0&0\cr
    [x_2]& 0&1&1&0&-1&0\cr
    [x_3]& 0&0&0&0&1&-1\cr
    [x_4]& 0&0&0&1&0&1\cr
  }
\end{displaymath}
\begin{displaymath}
 \partial_2=
  \bordermatrix{
    ~&[x_0,x_1,x_2]\cr
    [x_0,x_1]&1\cr
    [x_0,x_2]&-1\cr
    [x_1,x_2]&1\cr
    [x_1,x_4]&0\cr
    [x_2,x_3]&0\cr
    [x_3,x_4]&0\cr
  }
\end{displaymath}

As its name indicates, the boundary map applied to a linear
combination of simplices gives its boundary. The boundary of a
boundary is the null application. Therefore this theorem can be easily
demonstrated (see~\cite{hatcher_algebraic_2002} for instance):
\begin{theorem}
  For any $k$ integer, 
$$\partial_k \circ\partial_{k+1}=0.$$
\end{theorem}

Let $S$ be an abstract simplicial complex. Then we can denote 
the $k$-th boundary group of $S$ as $B_k(S)=\mathrm{im}
  \, \partial_{k+1}$, and the $k$-th cycle group of $S$ as
  $Z_k(S)=\ker \partial_{k}$. Then, we have $B_k(S)\subset Z_k(S)$

We are now able to define the $k$-th homology group:
\begin{definition}
  The $k$-th homology group of an abstract simplicial complex $S$ is
  the quotient vector space: 
  \begin{eqnarray*}
    H_k(S)=\frac{Z_k(S)}{B_k(S)}.
  \end{eqnarray*}
\end{definition}

Then its dimension is:
\begin{definition}
  The $k$-th Betti number of the abstract simplicial complex $S$ is: 
  \begin{eqnarray*}
    \beta_k(S)=\dim H_k(S).
  \end{eqnarray*}
\end{definition}

According to its definition, the $k$-th Betti number counts the number
of cycles of $k$-simplices that are not boundaries of
$(k+1)$-simplices, that are the $k$-th dimensional holes. In small
dimensions,  they have a geometrical interpretation:
\begin{itemize}
\item $\beta_0$ is the number of connected components,
\item $\beta_1$ is the number of coverage holes,
\item $\beta_2$ is the number of $3$D-voids.
\end{itemize}
For any $k\geq d$ where $d$ is the dimension, we can note that $\beta_k=0$.

We can compute the Betti numbers of the abstract simplicial complex of
Fig.~\ref{fig_complex}:
\begin{eqnarray*}
  \beta_0&=& \dim \ker \partial_0 - \dim \mathrm{im} \, \partial_1\\
  &=&1\\
 \beta_1&=& \dim \ker \partial_1 - \dim \mathrm{im} \, \partial_2\\
  &=& 1.
\end{eqnarray*}
This complex indeed has one connected component and one coverage hole
with $4$ sides.

For further reading on algebraic topology, see
\cite{hatcher_algebraic_2002}. 

\section{Reduction Algorithm}
\label{sec_ra}
In wireless networks, redundancy is frequent: sensors are cheap
devices, adding too many sensors to a network creates reliability with
reasonable cost. In cellular networks, dimensionning is done based on
peak traffic hours, and is thus under-used during low-traffic
periods. In order to conserve energy in both types of networks, some
nodes can be turned off, may it be temporarily or until other nodes
fail. But the topology of the network has to be maintained:
connectivity is needed for the network nodes to communicate, and
coverage maintenance assures the service to users. Simplicial homology
representation of wireless networks provides a mathematical
translation of the problem: remove vertices from an abstract
simplicial complex without modifying its topology.

The main idea of our algorithm is to use the information from the
topology of the network to reduce the number of vertices. First we use
simplicial homology representation to compute the topology of the
wireless network. Thanks to that representation, we are able to detect
the vertices that are the more redundant. We then remove vertices in
an optimal order for the computation complexity, while the topology of
the network, and in particular its Betti number, is unchanged. In the
remainder of this section, we present in full details the reduction
algorithm, that was first introduced in a limited form in
\cite{vergne_reduction_2013}. 

\subsection{Preliminaries}
The reduction algorithm takes as input an abstract simplicial complex
described by its list of simplices. But it also needs another
information. Indeed, if we consider for example an abstract simplicial
complex connected and without coverage holes, i.e.\ with Betti numbers
$\beta_0=1$ and $\beta_1=0$, the optimal mathematical reduction of
this complex is a single vertex. Its topology is in fact unchanged,
there is still one connected component, and no cycle around any
coverage hole. But that is not what we intend to do. Therefore we must
designate \emph{critical} vertices that define the limits of the
reduction. They usually define the boundary of the area, they
are then external boundary vertices. If there is a coverage hole whose
size must not be increased, the vertices surrounding it, that are
internal boundary vertices, have also to be listed as critical. But if
the area covered is not essential, then critical vertices can be
limited to access end points that have to stay connected for
instance. In 3D, the critical vertices would define the limit surface
of the volume. These critical vertices can not be removed by the algorithm.
Then the list of critical vertices, whatever they are, internal or
external boundary, access end points, etc, is given as input to the
reduction algorithm, along with the abstract simplicial complex. 
Please note that for the algorithm to give adequate results, the
critical vertices have to be correctly defined, that means that they
must really define the limits of the abstract simplicial complex to
reduce.

Then if the abstract simplicial complex is defined in dimension $d$,
then there exists $d$ nonzero homology groups. For $k\geq d$, the
$k$-th homology group $H_k$ does not exist and $\beta_k=0$. Therefore
it is possible to maintain the homology of the complex up to the
$(d-1)$-th degree. At the beginning of the algorithm, we must choose to
which degree the homology has to be maintained. For example, in two
dimensions, it is possible to maintain both connectivity and coverage,
$H_0$ and $H_1$,
or to maintain only the connectivity ($H_0$). Note that there is no sense to
maintain the $k$-th degree homology if the $(k-1)$-th is not
maintained since the former implies the latter, consider maintaining
coverage without connectivity for example. The degree of homology that
is to be maintained by the algorithm is denoted by $k_0$, typically in
two dimensions $k_0=1$, or sometimes $k_0=0$ to just maintain
connectivity without considering coverage.

\subsection{Characteristics} %definitions
The algorithm works by calculating some characteristics for the
simplices of the abstract simplicial complex it has to reduce. If the
degree of homology that has to be maintained is $k_0$, then the
largest size of simplices that are concerned are $(k_0+1)$-simplices
(in order to compute $B_{k_0} = \mathrm{im} \, \partial_{k_0+1}$). We
can conclude then that simplices larger than $(k_0+1)$-simplices are
useless to the homology up to the $k_0$-th degree. For instance,
in two dimensions, if we want to maintain the homology up to the first
degree, that is the coverage, only $2$-simplices are concerned, and
larger simplices are useless.

Our idea is now to sort the $(k_0+1)$-simplices in order to know which
ones to conserve in the reduced abstract simplicial complex. To do
that we characterize the superfluousness of each $(k_0+1)$-simplex via
a \emph{degree} that we define below:
\begin{definition}[Degree]
  Let $[x_0,\dots,x_k]$ be a $k$-simplex of an abstract simplicial
  complex $S$ for $k$ integer. Its degree is the size of its largest
  coface:
 \begin{eqnarray*}
    D[x_0,\dots,x_k]=\max\{ d \mid [x_0\dots,x_k] \subset d\text{-simplex}\}.
  \end{eqnarray*}
\end{definition}
By definition, we can see that for any $k$-simplex, $k$ integer,
$D[x_0,\dots,x_k] \geq k$. 

We can see an example of computation of
degrees for the $2$-simplices of an abstract simplicial complex in
Fig.~\ref{fig_degree}. A $2$-simplex that is a maximum face, as
$[x_2,x_6,x_7]$, has a degree of $2$, whereas a $2$-simplex that is
the face of a $4$-simplex, as $[x_0,x_1,x_7]$, has a degree of $4$. 
\begin{figure}[h]
  \centering
      \begin{tikzpicture}
\coordinate (x1) at (-1,0);
\coordinate (x2) at  (1,-0.2);
\coordinate (x3) at  (2.5,1.1);
\coordinate (x4) at (1.2,1.5);
\coordinate (x5) at (-0.5,1.5);
\coordinate (x6) at (0,-0.5);
\coordinate (x7) at (2.3,-0.2);
\coordinate (x8) at (3.2,0.9);
\coordinate (x9) at (3.3,-0.1);
\fill [color=clair] (x1)--(x4)--(x5);
\fill [color=clair] (x2)--(x6)--(x4);
\fill [color=moyen,opacity=0.4] (x2)--(x3)--(x4);
\fill [color=redbox,opacity=0.4] (x1)--(x6)--(x4);
\fill [color=clair] (x3)--(x9)--(x8);
\fill [color=apricot,opacity=0.4] (x7)--(x9)--(x3);
\draw[color=moyen] (x1)--(x2)--(x3)--(x4)--(x5)--(x1);
\draw[color=moyen] (x3)--(x7)--(x8)--(x9)--(x3)--(x8);
\draw[color=moyen] (x7)--(x9);
\draw[color=moyen] (x1)--(x4)--(x2)--(x5);
\draw[color=moyen] (x1)--(x6)--(x2);
\draw[color=moyen] (x4)--(x6)--(x5);
\fill [color=fonce] (x1) circle (2pt);
\fill [color=fonce] (x2) circle (2pt);
\fill [color=fonce] (x3) circle (2pt);
\fill [color=fonce] (x4) circle (2pt);
\fill [color=fonce] (x5) circle (2pt);
\fill [color=fonce] (x6) circle (2pt);
\fill [color=fonce] (x7) circle (2pt);
\fill [color=fonce] (x8) circle (2pt);
\fill [color=fonce] (x9) circle (2pt);
\node [below] at (x1) {$x_0$};
\node [below] at (x2) {$x_2$};
\node [below] at (x6) {$x_1$};
\node [above] at (x3) {$x_6$};
\node [above] at (x4) {$x_7$};
\node [above] at (x5) {$x_8$};
\node [below] at (x7) {$x_3$};
\node [above] at (x8) {$x_5$};
\node [below] at (x9) {$x_4$};
\node [below] at (-0.2,-1) {\color{redbox}$D[x_0,x_1,x_7]=4$};
\node [above] at (1.75,1.9) {\color{moyen}$D[x_2,x_6,x_7]=2$};
\node [below] at (2.8,-0.8) {\color{apricot}$D[x_3,x_4,x_6]=3$};
    \end{tikzpicture}
  \caption{Example of computation of degrees of $2$-simplices}
\label{fig_degree}
\end{figure}
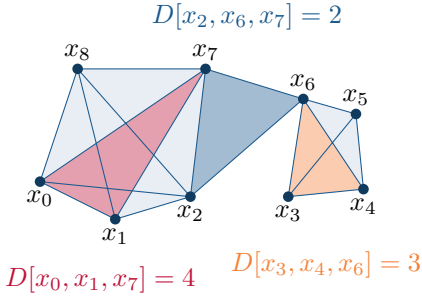

The greater the degree of a $k$-simplex is, the more superfluous this
$k$-simplex is. Therefore, the greater the degree of a
$(k_0+1)$-simplex is, the more likely it is to be removed, and its
removal is less likely to modify the $k_0$-th homology.

But, in order to reduce an abstract simplicial complex, we do not
remove directly $(k_0+1)$-simplices, we remove $0$-simplices, that
lead to the removal of their cofaces. Therefore, we need to bring the
information about the superfluousness of a $(k_0+1)$-simplex that is
contained in its degree, down to the $0$-simplex level. To do that, we
define an index for every $0$-simplex, that characterizes its level of
sensitivity for the $k_0$-th homology.
\begin{definition}[Index]
  Let $x$ be a $0$-simplex, its index is the minimum of its
  $(k_0+1)$-cofaces degrees:
  \begin{eqnarray*}
    I(x)=\min\{ D[x_0,\dots,x_{k_0+1}] \mid x \in [x_0,\dots,x_{k_0+1}]\}.
  \end{eqnarray*}
If $x$ has no $(k_0+1)$-coface then $I(x)=0$.
\end{definition}

We can see an example in Fig.~\ref{fig_index} of computation of indices
of the $0$-simplices of the abstract simplicial complex of
Fig.~\ref{fig_degree}. In this example, we are interested in the
$0$-th and $1$-st degrees of homology, i.e.\ connectivity and coverage,
that means that $k_0=1$ and we compute the degrees on the
$2$-simplices. 
\begin{figure}[h]
  \centering
        \begin{tikzpicture}
\coordinate (x1) at (-1,0);
\coordinate (x2) at  (1,-0.2);
\coordinate (x3) at  (2.5,1.1);
\coordinate (x4) at (1.2,1.5);
\coordinate (x5) at (-0.5,1.5);
\coordinate (x6) at (0,-0.5);
\coordinate (x7) at (2.3,-0.2);
\coordinate (x8) at (3.2,0.9);
\coordinate (x9) at (3.3,-0.1);
\fill [color=clair] (x1)--(x4)--(x5);
\fill [color=clair] (x2)--(x6)--(x4);
\fill [color=clair] (x2)--(x3)--(x4);
\fill [color=clair] (x1)--(x6)--(x4);
\fill [color=clair] (x3)--(x9)--(x8);
\fill [color=clair] (x7)--(x9)--(x3);
\draw[color=moyen] (x1)--(x2)--(x3)--(x4)--(x5)--(x1);
\draw[color=moyen] (x3)--(x7)--(x8)--(x9)--(x3)--(x8);
\draw[color=moyen] (x7)--(x9);
\draw[color=moyen] (x1)--(x4)--(x2)--(x5);
\draw[color=moyen] (x1)--(x6)--(x2);
\draw[color=moyen] (x4)--(x6)--(x5);
\fill [color=redbox] (x1) circle (2pt);
\fill [color=moyen] (x2) circle (2pt);
\fill [color=moyen] (x3) circle (2pt);
\fill [color=moyen] (x4) circle (2pt);
\fill [color=redbox] (x5) circle (2pt);
\fill [color=redbox] (x6) circle (2pt);
\fill [color=apricot] (x7) circle (2pt);
\fill [color=apricot] (x8) circle (2pt);
\fill [color=apricot] (x9) circle (2pt);
\node [below] at (-1.5,-0.2) {\color{redbox}$I(x_0)=4$};
\node [below] at (-0.5,-0.7) {\color{redbox}$I(x_1)=4$};
\node [below] at (1,-0.4) {\color{moyen}$I(x_2)=2$};
\node [below] at (2.5,-0.6) {\color{apricot}$I(x_3)=3$};
\node [below] at (4,-0.4) {\color{apricot}$I(x_4)=3$};
\node [above] at (4,1.1) {\color{apricot}$I(x_5)=3$};
\node [above] at (2.5,1.4) {\color{moyen}$I(x_6)=2$};
\node [above] at (0.9,1.7) {\color{moyen}$I(x_7)=2$};
\node [above] at (-1,1.7) {\color{redbox}$I(x_8)=4$};
\node [below] at (x1) {$x_0$};
\node [below] at (x2) {$x_2$};
\node [below] at (x6) {$x_1$};
\node [above] at (x3) {$x_6$};
\node [above] at (x4) {$x_7$};
\node [above] at (x5) {$x_8$};
\node [below] at (x7) {$x_3$};
\node [above] at (x8) {$x_5$};
\node [below] at (x9) {$x_4$};
    \end{tikzpicture}
  \caption{Example of computation of indices of $0$-simplices}
\label{fig_index}
\end{figure}
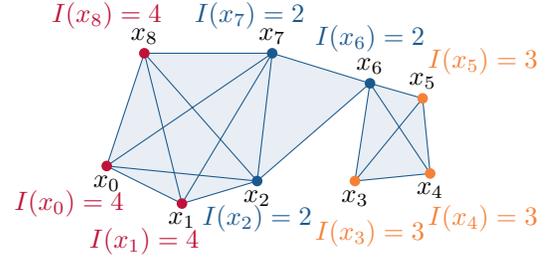

With this definition of indices, we can see that a $0$-simplex is as sensitive as
its most sensitive coface, or in other terms as superfluous as its least superfluous
coface. The index of a $0$-simplex can also be viewed as an indicator
of the density of $0$-simplices ``around'' it, in the neighbor
sense. For example, an index of $k_0+1$ indicates that at least one of
its $(k_0+1)$-coface has no $(k_0+2)$-cofaces, whereas an index
of $k >k_0+1$ indicates that each of its $(k_0+1)$-cofaces are the face of
simplices larger than $k$-simplices. The main idea of the algorithm is
thus to remove the vertices with the greatest indices. 

\subsection{Algorithm description}
As stated in the previous subsections, the reduction algorithm takes
as input an abstract simplicial complex and the list of critical
$0$-simplices that can not be removed. We must also know
the degree $k_0$ of homology that has to be maintained.
The algorithm begins by computing the topology of the network up to the
$k_0$-th degree of homology, that is the first $k_0+1$ Betti
numbers. Then, it computes the characteristics we need: the degrees of
the $(k_0+1)$-simplices and the indices of the $0$-simplices. After
that, the algorithm removes a $0$-simplex with a maximal index. If
there are more than one such index, one is chosen randomly
uniformly among them. The removal of a $0$-simplex leads to the
removal of all of its cofaces. Then the algorithm goes on doing the
same thing with the obtained reduced abstract simplicial complex.

There is one exception: if there is a difference between the Betti numbers
computed on the newly reduced abstract simplicial complex and the
original ones, the removal of the $0$-simplex is cancelled. That means that this
$0$-simplex is put back in the simplicial complex. To prevent from
trying to removing it again, that would lead to the same conclusion,
the $0$-simplex is flagged as critical, along with the input critical
$0$-simplices. Thus the list of critical $0$-simplices evolve during the
algorithm execution.

Then the algorithm goes on repeating the computation of the Betti
numbers, degrees and indices to find a $0$-simplex to remove. It
removes $0$-simplices one by one until the maximum index of a $0$-simplex
is equal to $k_0+1$. Indeed it is the minimum value for the degrees of
$(k_0+1)$-simplices, so it is the lower bound that it can reach.

We give in Algorithm~\ref{alg_ra} the whole reduction algorithm for the conservation
of the $k_0$-th homology. We use a negative index equal to $-1$ to flag
critical $0$-simplices as such. 

Considering an abstract simplicial
complex $S$, we denote by $s_k(S)$ its number of $k$-simplices and by
$\beta_k(S)$ its $k$-th Betti number, for $k$ integer. Then we denote by 
$x_1, \dots, x_{s_0(S)}$ its $0$-simplices and by $y_1,\dots, y_{s_{k_0+1}(S)}$
its $(k_0+1)$-simplices. We write $I(x)$ for the index of the $0$-simplex
$x$, and $D(y)$ for the degree of the $(k_0+1)$-simplex $y$.

\begin{algorithm}[h]
  \caption{Reduction algorithm}
\label{alg_ra}
  \begin{algorithmic}[h]
    \REQUIRE{abstract simplicial complex $S$,
 list $L$ of critical $0$-simplices.}
\STATE{Computation of $\beta_0(S), \dots, \beta_{k_0}(S)$\;}
\STATE{Computation of $D(y_1),\dots, D(y_{s_{k_0+1}(S)})$\;}
\STATE{Computation of $I(x_1),\dots,I(x_{s_0(S)})$\;}
\FORALL{$ x \in L$} 
\STATE{$I(x)=-1$\;}
\ENDFOR{}
 \STATE{$I_{\max}=\max \{I(x_1),\dots,I(x_{s_0(S)})\}$\;}
\WHILE{$I_{\max} > k_0+1$} 
\STATE{Draw uniformly $\hat x$ such that $I(\hat x)=I_{\max}$\;}
\STATE{$S'=S\backslash \{\hat x\}$\; \%Removal of $\hat x$ and its
cofaces}
\STATE{Computation of $\beta_0(S'), \dots, \beta_{k_0}(S')$\;}
\IF{$\beta_k(S')\neq \beta_k(S) \text{ for some } k=0,\dots, k_0$}
    \STATE{$I(\hat x)=-1$\;}
\ELSE{} 
\STATE{Computation of $D(y_1),\dots, D(y_{s_{k_0+1}(S')})$\;}
\FORALL{$ x \in \{x_1,\dots,x_{s_0(S')}\}$}
    \IF{$I(x)\neq-1$}
\STATE{Computation of $I(x)$\;}
\ENDIF{}
\ENDFOR{}
\STATE{$I_{\max}=\max \{I(x_1),\dots,I(x_{s_0(S')})\}$\;}
\STATE{$S=S'$\;}
    \ENDIF{}
    \ENDWHILE{}
    \RETURN{$X$}
  \end{algorithmic}
\end{algorithm}

We can see in Fig.~\ref{fig_reduc} an example of the reduction
algorithm for the homology conservation up to the first degree on a
Vietoris-Rips complex with a boundary of critical vertices along the
square. 

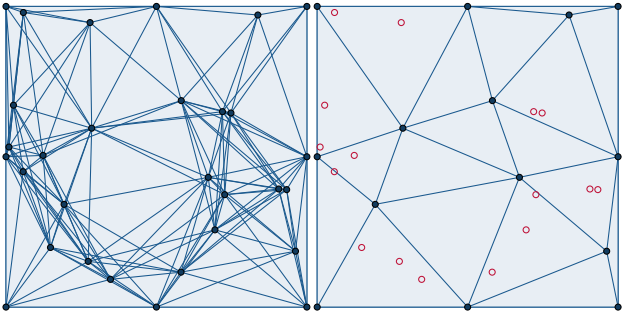
\begin{figure}[h]
  \centering
\begin{tikzpicture}[scale=2]
\draw  [draw=moyen] (0,0) rectangle (2,2); 
\fill [fill=clair] (0,0) rectangle (2,2);
\useasboundingbox (0,0) rectangle (2,2);
\begin{pgfonlayer}{foreground layer}
\coordinate (x1) at (1,0);
\coordinate (x2) at (2,0);
\coordinate (x3) at (2,1);
\coordinate (x4) at (2,2);
\coordinate (x5) at (1,2);
\coordinate (x6) at (0,2);
\coordinate (x7) at (0,1);
\coordinate (x8) at (0,0);
\coordinate (x9) at (1.4953,1.2911);
\coordinate (x10) at (0.2464,1.0088);
\coordinate (x11) at (0.6945,0.1843);
\coordinate (x12) at (0.2957,0.3963);
\coordinate (x13) at (1.3445,0.8630);
\coordinate (x14) at (1.3888,0.5136);
\coordinate (x15) at (0.0195,1.0646);
\coordinate (x16) at (0.5588,1.8925);
\coordinate (x17) at (1.8129,0.7854);
\coordinate (x18) at (0.0497,1.3429);
\coordinate (x19) at (1.6743,1.9430);
\coordinate (x20) at (0.1139,0.9006);
\coordinate (x21) at (1.1649,1.3733);
\coordinate (x22) at (1.4389,1.3001);
\coordinate (x23) at (1.4538,0.7477);
\coordinate (x24) at (1.1632,0.2322);
\coordinate (x25) at (0.1153,1.9595);
\coordinate (x26) at (0.5696,1.1899);
\coordinate (x27) at (1.9243,0.3716);
\coordinate (x28) at (0.3861,0.6833);
\coordinate (x29) at (1.8656,0.7813);
\coordinate (x30) at (0.5464,0.3039);
\draw [color=moyen]  (x1)--(x2);
\draw [color=moyen]  (x1)--(x8);
\draw [color=moyen]  (x1)--(x11);
\draw [color=moyen]  (x1)--(x12);
\draw [color=moyen]  (x1)--(x13);
\draw [color=moyen]  (x1)--(x14);
\draw [color=moyen]  (x1)--(x23);
\draw [color=moyen]  (x1)--(x24);
\draw [color=moyen]  (x1)--(x27);
\draw [color=moyen]  (x1)--(x28);
\draw [color=moyen]  (x1)--(x30);
\draw [color=moyen]  (x2)--(x3);
\draw [color=moyen]  (x2)--(x14);
\draw [color=moyen]  (x2)--(x17);
\draw [color=moyen]  (x2)--(x23);
\draw [color=moyen]  (x2)--(x24);
\draw [color=moyen]  (x2)--(x27);
\draw [color=moyen]  (x2)--(x29);
\draw [color=moyen]  (x3)--(x4);
\draw [color=moyen]  (x3)--(x9);
\draw [color=moyen]  (x3)--(x13);
\draw [color=moyen]  (x3)--(x14);
\draw [color=moyen]  (x3)--(x17);
\draw [color=moyen]  (x3)--(x19);
\draw [color=moyen]  (x3)--(x21);
\draw [color=moyen]  (x3)--(x22);
\draw [color=moyen]  (x3)--(x23);
\draw [color=moyen]  (x3)--(x27);
\draw [color=moyen]  (x3)--(x29);
\draw [color=moyen]  (x4)--(x5);
\draw [color=moyen]  (x4)--(x9);
\draw [color=moyen]  (x4)--(x19);
\draw [color=moyen]  (x4)--(x22);
\draw [color=moyen]  (x5)--(x6);
\draw [color=moyen]  (x5)--(x9);
\draw [color=moyen]  (x5)--(x16);
\draw [color=moyen]  (x5)--(x19);
\draw [color=moyen]  (x5)--(x21);
\draw [color=moyen]  (x5)--(x22);
\draw [color=moyen]  (x5)--(x25);
\draw [color=moyen]  (x5)--(x26);
\draw [color=moyen]  (x6)--(x7);
\draw [color=moyen]  (x6)--(x15);
\draw [color=moyen]  (x6)--(x16);
\draw [color=moyen]  (x6)--(x18);
\draw [color=moyen]  (x6)--(x25);
\draw [color=moyen]  (x6)--(x26);
\draw [color=moyen]  (x7)--(x8);
\draw [color=moyen]  (x7)--(x10);
\draw [color=moyen]  (x7)--(x12);
\draw [color=moyen]  (x7)--(x15);
\draw [color=moyen]  (x7)--(x18);
\draw [color=moyen]  (x7)--(x20);
\draw [color=moyen]  (x7)--(x25);
\draw [color=moyen]  (x7)--(x26);
\draw [color=moyen]  (x7)--(x28);
\draw [color=moyen]  (x7)--(x30);
\draw [color=moyen]  (x8)--(x11);
\draw [color=moyen]  (x8)--(x12);
\draw [color=moyen]  (x8)--(x20);
\draw [color=moyen]  (x8)--(x28);
\draw [color=moyen]  (x8)--(x30);
\draw [color=moyen]  (x9)--(x13);
\draw [color=moyen]  (x9)--(x14);
\draw [color=moyen]  (x9)--(x17);
\draw [color=moyen]  (x9)--(x19);
\draw [color=moyen]  (x9)--(x21);
\draw [color=moyen]  (x9)--(x22);
\draw [color=moyen]  (x9)--(x23);
\draw [color=moyen]  (x9)--(x26);
\draw [color=moyen]  (x9)--(x29);
\draw [color=moyen]  (x10)--(x11);
\draw [color=moyen]  (x10)--(x12);
\draw [color=moyen]  (x10)--(x15);
\draw [color=moyen]  (x10)--(x16);
\draw [color=moyen]  (x10)--(x18);
\draw [color=moyen]  (x10)--(x20);
\draw [color=moyen]  (x10)--(x21);
\draw [color=moyen]  (x10)--(x25);
\draw [color=moyen]  (x10)--(x26);
\draw [color=moyen]  (x10)--(x28);
\draw [color=moyen]  (x10)--(x30);
\draw [color=moyen]  (x11)--(x12);
\draw [color=moyen]  (x11)--(x13);
\draw [color=moyen]  (x11)--(x14);
\draw [color=moyen]  (x11)--(x20);
\draw [color=moyen]  (x11)--(x23);
\draw [color=moyen]  (x11)--(x24);
\draw [color=moyen]  (x11)--(x28);
\draw [color=moyen]  (x11)--(x30);
\draw [color=moyen]  (x12)--(x15);
\draw [color=moyen]  (x12)--(x18);
\draw [color=moyen]  (x12)--(x20);
\draw [color=moyen]  (x12)--(x24);
\draw [color=moyen]  (x12)--(x26);
\draw [color=moyen]  (x12)--(x28);
\draw [color=moyen]  (x12)--(x30);
\draw [color=moyen]  (x13)--(x14);
\draw [color=moyen]  (x13)--(x17);
\draw [color=moyen]  (x13)--(x21);
\draw [color=moyen]  (x13)--(x22);
\draw [color=moyen]  (x13)--(x23);
\draw [color=moyen]  (x13)--(x24);
\draw [color=moyen]  (x13)--(x26);
\draw [color=moyen]  (x13)--(x27);
\draw [color=moyen]  (x13)--(x28);
\draw [color=moyen]  (x13)--(x29);
\draw [color=moyen]  (x13)--(x30);
\draw [color=moyen]  (x14)--(x17);
\draw [color=moyen]  (x14)--(x21);
\draw [color=moyen]  (x14)--(x22);
\draw [color=moyen]  (x14)--(x23);
\draw [color=moyen]  (x14)--(x24);
\draw [color=moyen]  (x14)--(x27);
\draw [color=moyen]  (x14)--(x29);
\draw [color=moyen]  (x14)--(x30);
\draw [color=moyen]  (x15)--(x16);
\draw [color=moyen]  (x15)--(x18);
\draw [color=moyen]  (x15)--(x20);
\draw [color=moyen]  (x15)--(x25);
\draw [color=moyen]  (x15)--(x26);
\draw [color=moyen]  (x15)--(x28);
\draw [color=moyen]  (x15)--(x30);
\draw [color=moyen]  (x16)--(x18);
\draw [color=moyen]  (x16)--(x21);
\draw [color=moyen]  (x16)--(x25);
\draw [color=moyen]  (x16)--(x26);
\draw [color=moyen]  (x17)--(x21);
\draw [color=moyen]  (x17)--(x22);
\draw [color=moyen]  (x17)--(x23);
\draw [color=moyen]  (x17)--(x24);
\draw [color=moyen]  (x17)--(x27);
\draw [color=moyen]  (x17)--(x29);
\draw [color=moyen]  (x18)--(x20);
\draw [color=moyen]  (x18)--(x25);
\draw [color=moyen]  (x18)--(x26);
\draw [color=moyen]  (x18)--(x28);
\draw [color=moyen]  (x19)--(x21);
\draw [color=moyen]  (x19)--(x22);
\draw [color=moyen]  (x20)--(x26);
\draw [color=moyen]  (x20)--(x28);
\draw [color=moyen]  (x20)--(x30);
\draw [color=moyen]  (x21)--(x22);
\draw [color=moyen]  (x21)--(x23);
\draw [color=moyen]  (x21)--(x26);
\draw [color=moyen]  (x21)--(x29);
\draw [color=moyen]  (x22)--(x23);
\draw [color=moyen]  (x22)--(x26);
\draw [color=moyen]  (x22)--(x29);
\draw [color=moyen]  (x23)--(x24);
\draw [color=moyen]  (x23)--(x26);
\draw [color=moyen]  (x23)--(x27);
\draw [color=moyen]  (x23)--(x29);
\draw [color=moyen]  (x24)--(x27);
\draw [color=moyen]  (x24)--(x28);
\draw [color=moyen]  (x24)--(x29);
\draw [color=moyen]  (x24)--(x30);
\draw [color=moyen]  (x25)--(x26);
\draw [color=moyen]  (x26)--(x28);
\draw [color=moyen]  (x26)--(x30);
\draw [color=moyen]  (x27)--(x29);
\draw [color=moyen]  (x28)--(x30);
\draw [fill=fonce] (x1) circle (0.02cm);
\draw [fill=fonce] (x2) circle (0.02cm);
\draw [fill=fonce] (x3) circle (0.02cm);
\draw [fill=fonce] (x4) circle (0.02cm);
\draw [fill=fonce] (x5) circle (0.02cm);
\draw [fill=fonce] (x6) circle (0.02cm);
\draw [fill=fonce] (x7) circle (0.02cm);
\draw [fill=fonce] (x8) circle (0.02cm);
\draw [fill=fonce] (x9) circle (0.02cm);
\draw [fill=fonce] (x10) circle (0.02cm);
\draw [fill=fonce] (x11) circle (0.02cm);
\draw [fill=fonce] (x12) circle (0.02cm);
\draw [fill=fonce] (x13) circle (0.02cm);
\draw [fill=fonce] (x14) circle (0.02cm);
\draw [fill=fonce] (x15) circle (0.02cm);
\draw [fill=fonce] (x16) circle (0.02cm);
\draw [fill=fonce] (x17) circle (0.02cm);
\draw [fill=fonce] (x18) circle (0.02cm);
\draw [fill=fonce] (x19) circle (0.02cm);
\draw [fill=fonce] (x20) circle (0.02cm);
\draw [fill=fonce] (x21) circle (0.02cm);
\draw [fill=fonce] (x22) circle (0.02cm);
\draw [fill=fonce] (x23) circle (0.02cm);
\draw [fill=fonce] (x24) circle (0.02cm);
\draw [fill=fonce] (x25) circle (0.02cm);
\draw [fill=fonce] (x26) circle (0.02cm);
\draw [fill=fonce] (x27) circle (0.02cm);
\draw [fill=fonce] (x28) circle (0.02cm);
\draw [fill=fonce] (x29) circle (0.02cm);
\draw [fill=fonce] (x30) circle (0.02cm);
\end{pgfonlayer}{foreground layer}
\end{tikzpicture} 
\begin{tikzpicture}[scale=2]
\draw  [draw=moyen] (0,0) rectangle (2,2); 
\fill [fill=clair] (0,0) rectangle (2,2);
\useasboundingbox (0,0) rectangle (2,2);
\begin{pgfonlayer}{foreground layer}
\coordinate (x1) at (1,0);
\coordinate (x2) at (2,0);
\coordinate (x3) at (2,1);
\coordinate (x4) at (2,2);
\coordinate (x5) at (1,2);
\coordinate (x6) at (0,2);
\coordinate (x7) at (0,1);
\coordinate (x8) at (0,0);
\coordinate (x9) at (1.4953,1.2911);
\coordinate (x10) at (0.2464,1.0088);
\coordinate (x11) at (0.6945,0.1843);
\coordinate (x12) at (0.2957,0.3963);
\coordinate (x13) at (1.3445,0.8630);
\coordinate (x14) at (1.3888,0.5136);
\coordinate (x15) at (0.0195,1.0646);
\coordinate (x16) at (0.5588,1.8925);
\coordinate (x17) at (1.8129,0.7854);
\coordinate (x18) at (0.0497,1.3429);
\coordinate (x19) at (1.6743,1.9430);
\coordinate (x20) at (0.1139,0.9006);
\coordinate (x21) at (1.1649,1.3733);
\coordinate (x22) at (1.4389,1.3001);
\coordinate (x23) at (1.4538,0.7477);
\coordinate (x24) at (1.1632,0.2322);
\coordinate (x25) at (0.1153,1.9595);
\coordinate (x26) at (0.5696,1.1899);
\coordinate (x27) at (1.9243,0.3716);
\coordinate (x28) at (0.3861,0.6833);
\coordinate (x29) at (1.8656,0.7813);
\coordinate (x30) at (0.5464,0.3039);
\draw [color=moyen]  (x1)--(x2);
\draw [color=moyen]  (x1)--(x8);
\draw [color=moyen]  (x1)--(x13);
\draw [color=moyen]  (x1)--(x27);
\draw [color=moyen]  (x1)--(x28);
\draw [color=moyen]  (x2)--(x3);
\draw [color=moyen]  (x2)--(x27);
\draw [color=moyen]  (x3)--(x4);
\draw [color=moyen]  (x3)--(x13);
\draw [color=moyen]  (x3)--(x19);
\draw [color=moyen]  (x3)--(x21);
\draw [color=moyen]  (x3)--(x27);
\draw [color=moyen]  (x4)--(x5);
\draw [color=moyen]  (x4)--(x19);
\draw [color=moyen]  (x5)--(x6);
\draw [color=moyen]  (x5)--(x19);
\draw [color=moyen]  (x5)--(x21);
\draw [color=moyen]  (x5)--(x26);
\draw [color=moyen]  (x6)--(x7);
\draw [color=moyen]  (x6)--(x26);
\draw [color=moyen]  (x7)--(x8);
\draw [color=moyen]  (x7)--(x26);
\draw [color=moyen]  (x7)--(x28);
\draw [color=moyen]  (x8)--(x28);
\draw [color=moyen]  (x13)--(x21);
\draw [color=moyen]  (x13)--(x26);
\draw [color=moyen]  (x13)--(x27);
\draw [color=moyen]  (x13)--(x28);
\draw [color=moyen]  (x19)--(x21);
\draw [color=moyen]  (x21)--(x26);
\draw [color=moyen]  (x26)--(x28);
\draw [fill=fonce] (x1) circle (0.02cm);
\draw [fill=fonce] (x2) circle (0.02cm);
\draw [fill=fonce] (x3) circle (0.02cm);
\draw [fill=fonce] (x4) circle (0.02cm);
\draw [fill=fonce] (x5) circle (0.02cm);
\draw [fill=fonce] (x6) circle (0.02cm);
\draw [fill=fonce] (x7) circle (0.02cm);
\draw [fill=fonce] (x8) circle (0.02cm);
\draw [draw=redbox] (x9) circle (0.02cm);
\draw [draw=redbox] (x10) circle (0.02cm);
\draw [draw=redbox] (x11) circle (0.02cm);
\draw [draw=redbox] (x12) circle (0.02cm);
\draw [fill=fonce] (x13) circle (0.02cm);
\draw [draw=redbox] (x14) circle (0.02cm);
\draw [draw=redbox] (x15) circle (0.02cm);
\draw [draw=redbox] (x16) circle (0.02cm);
\draw [draw=redbox] (x17) circle (0.02cm);
\draw [draw=redbox] (x18) circle (0.02cm);
\draw [fill=fonce] (x19) circle (0.02cm);
\draw [draw=redbox] (x20) circle (0.02cm);
\draw [fill=fonce] (x21) circle (0.02cm);
\draw [draw=redbox] (x22) circle (0.02cm);
\draw [draw=redbox] (x23) circle (0.02cm);
\draw [draw=redbox] (x24) circle (0.02cm);
\draw [draw=redbox] (x25) circle (0.02cm);
\draw [fill=fonce] (x26) circle (0.02cm);
\draw [fill=fonce] (x27) circle (0.02cm);
\draw [fill=fonce] (x28) circle (0.02cm);
\draw [draw=redbox] (x29) circle (0.02cm);
\draw [draw=redbox] (x30) circle (0.02cm);
\end{pgfonlayer}{foreground layer}
\end{tikzpicture} 
\caption{Example of the reduction algorithm on a Vietoris-Rips complex.}
\label{fig_reduc} 
\end{figure}

We can note that it is possible to reduce the computations of the
algorithm by remarking that when a $0$-simplex of index $I$ is removed
then only the $0$-faces of its $I$-cofaces with index $I$ can have their index
impacted as proofed in the following lemma. That means that we only
need to re-compute the degrees needed for the indices of the
$0$-simplices that  shared an $I$-simplex with the removed
$0$-simplex.  

\begin{lemma}
\label{lemma_I}
When a $0$-simplex of index $I$ is removed, only the
  $0$-faces of its $I$-cofaces with index equal to $I$ can have their
  index modified. 
\end{lemma}

\begin{proof}
  Let $\hat x$ be the removed $0$-simplex, and $x$ be any $0$-simplex of
  the current abstract simplicial complex, we denote by $I(x)$ its
  index. We differentiate four cases: 
  \begin{itemize}
  \item $x$ and $\hat x$ have no common coface.
  \end{itemize}
Then none of the degrees of the $(k_0+1)$-cofaces of $x$ will change,
and neither will its index.
  \begin{itemize}
  \item $x$ and $\hat x$ have a maximum common coface that is a
    $k$-simplex with $k<k_0$. 
  \end{itemize}
As in the previous case, none of the degrees of the $(k_0+1)$-cofaces of
$x$ will change, and neither will its index.
 \begin{itemize}
  \item $x$ and $\hat x$ have a maximum common coface that is a
    $k$-simplex with $k_0\leq k < I$.
  \end{itemize}
Then $\hat x$ should have an index of $k<I$, which is absurd.
\begin{itemize}
  \item $x$ and $\hat x$ have a maximum common coface that is a
    $k$-simplex with $k \geq I$.
  \end{itemize}
Either $I(x)<I$, then it comes from the degree of a $(k_0+1)$-simplex
not common with $\hat x$, and its index does not change with the
removal of $\hat x$.
Else, if $I(x)=I$, either its value comes from a $I$-simplex not shared
with $\hat x$ and remains unmodified. Or it comes from a common
$I$-simplex. Only in this latter case, $I(x)$ is modified by the
removal of $\hat x$. 
\end{proof}

\subsection{Properties}

\subsubsection{Homology invariance}
We have built our reduction algorithm to be homology invariant. That
means that the initial abstract complex, the final reduced complex and
every intermediary reduced complex are homotopy equivalent. They have
the same Betti numbers and any basis element of the $k$-th homology
group $H_k$ in the initial complex can be mapped to a basis element of
the $k$-th homology group in the final complex for any $k$ integer.

\begin{theorem}
The reduction algorithm stated in the algorithm described in
Alg.~\ref{alg_ra} is homology invariant up to the $k_0$-th degree.   
\end{theorem}
\begin{proof}
To verify that the algorithm is homology invariant up to the $k_0$-th
degree, we need only to check that each loop does not modify the
$k_0$-th homology. In each loop of the algorithm, we verify that the
Betti numbers $\beta_0, \dots, \beta_{k_0}$ are unchanged. That means
that the dimension of the homology groups $H_0, \dots, H_{k_0}$ do not
change. For any $k \in \{0,\dots,k_0\}$, if one or more cycle are
added in $H_k$, by the removal of a $0$-simplex, then the removal 
is cancelled by the algorithm. The same goes with one or more
deletions of cycles in $H_k$. The only way that a change in the homology is
undetected and allowed by the algorithm is if in a loop the same
number of cycles of $H_k$ are simultaneously added and deleted by a
single $0$-simplex removal. 

For any $k \in \{0,\dots,k_0\}$, $H_k$ is the $k$-th homology group of
cycles of $k$-simplices ($Z_k$) that are not boundaries of
$(k+1)$-simplices ($B_k$), note that $B_k \subset Z_k$.

On the one hand, let us look what happens if a cycle is added in
$H_k$. Since a new cycle of $Z_k$, that is a list of $k$-simplices, can not be
created by removing simplices, the removal of a $0$-simplex adds a
cycle in $H_k$ only if a cycle which was both in $B_k$ and $Z_k$
ceases to be in $B_k$, i.e.\ if the $0$-simplex had a $(k+1)$-coface
that was not redundant.

On the other hand, we investigate the deletion of a cycle in
$H_k$. The removal of a $0$-simplex can not make a cycle only in $Z_k$
to be in both $B_k$ and $Z_k$ since it does not create
$(k+1)$-simplices. So a deletion of a cycle in $H_k$, due to the
removal of a $0$-simplex, is necessarily a deletion in $Z_k$. 
Therefore, the removal of a  $0$-simplex deletes a cycle in $H_k$
only if this $0$-simplex had a $k$-coface which was in the cycles of
$Z_k$ and not the boundaries of $B_K$. That means that this
$0$-simplex was a boundary vertex of a $k$-th dimensional hole.

If a $0$-simplex that is both a boundary vertex to a $k$-th
dimensional hole and has a non-redundant $(k+1)$-coface is removed,
simultaneously a cycle is deleted and another is created in $H_k$.
The deletion of the $0$-simplex leads to an enlargement of the $k$-th
dimensional hole. The $(k+1)$-coface assures that the hole
still has boundaries and exists. It is then possible to map the cycle of
the smaller (previous to the $0$-simplex removal) hole to the larger
(post $0$-simplex removal) hole. And the two abstract simplicial
complexes (pre and post $0$-simplex removal) are homotopy equivalent. 
\end{proof}

We can note that if we want to have a reduction algorithm that is
homology invariant and that do not enlarge $k$-th dimensional holes,
then all the hole boundary vertices must be defined as critical in the input.

\subsubsection{Optimal order for vertices removal}
In our reduction algorithm, we choose to compute the topology and use
it to reduce an abstract simplicial complex in order to minimize the
size of the complex in relation to its topology. In addition to
keeping the homology invariant, we use the homology information to
improve the algorithm performance.

\begin{theorem}
The order in which $0$-simplices and their cofaces are removed from
the abstract simplicial complex by the reduction algorithm defined in
Alg.~\ref{alg_ra} is optimal for its computation complexity.
\end{theorem}

\begin{proof}
When a $0$-simplex is removed from the abstract simplicial
complex, all of its coface are subsequently removed too. In
Alg.~\ref{alg_ra}, the $0$-simplices are removed by decreasing
indices. An index of $I$ indicates that every $(k_0+1)$-coface of the
$0$-simplex has $I$-faces at least. The bigger the index of a
$0$-simplex is, the bigger the degrees of its $(k_0+1)$-cofaces
are. So the removal of the maximum index $0$-simplex minimize at most
the degrees computations that must be done in every loop.

Moreover we can see that the bigger the index of a $0$-simplex is, the
bigger the minimum common size of its cofaces greater than
$(k_0+1)$-simplices is. An abstract simplicial complex is implemented
by the list of all its simplices. And the homology up to the $k_0$-th
degree (i.e. $(k_0+1)$-simplices), is computed. By eliminating the
greatest $(k_0+1)$-simplices cofaces, the algorithm reduces also the
size of the simplicial complex implementation, and its topology
computation.  
\end{proof}

\subsubsection{Optimal solution}
The reduction algorithm reaches a local optimum, that may not be
the global optimum if there are multiple local optima. In game theory
vocabulary, that means that the algorithm reaches a Nash equilibrium
as defined in~\cite{campos-nanez_game-theoretic_2008}:
\begin{theorem}
The reduction algorithm defined in Alg.~\ref{alg_ra} reaches a Nash
equilibrium: the final complex can not be further reduced, no more
$0$-simplex can be removed.
\end{theorem}

\begin{proof}
By definition of degrees and indices given in the previous subsection,
indices computed on $(k_0+1)$-simplices degrees are greater or equal to
$k_0+1$. However, in the final complex, every $0$-simplex is of index
smaller or equal to $k_0+1$, since its the ending condition on the
``while'' loop. 

Then there are three possibility for the value $I(x)$ of the index of
a $0$-simplex $x$ in the final complex:
\begin{itemize}
\item $I(x)=-1$ which means that $x$ has been defined or flagged as
  critical.
\item $I(x)=0$ which means that $x$ has no $(k_0+1)$-coface.
\item $I(x)=k_0+1$.
\end{itemize}

First, a critical $0$-simplex can be either defined as such, in which case
its removal was forbidden, and it should be in the final complex. Or
it as been flagged by the algorithm because its removal was tried and
changed the Betti numbers of the abstract simplicial complex, which is
forbidden. If the initial critical $0$-simplices are well-defined as
the limits of the complex, then the removal of a vertex $0$-simplex
that has led to a change in a Betti number would always lead to the
same change. So a flagged critical $0$-simplex stays as such.

Secondly, a $0$-simplex of null index is an isolated vertex for
the $k_0$-th homology as it has no $(k_0+1)$-coface. Then, its removal 
would decrease one of the Betti numbers $\beta_0,\dots
,\beta_{k_0}$. For example, the removal of a $0$-simplex with no
$1$-cofaces will decrease $\beta_0$, and the removal of a $0$-simplex
with $1$-cofaces and no $2$-cofaces would decrease either $\beta_1$.

Finally, if $x$ has a degree of $I(x)=k_0+1$. That means that at least one
of its $(k_0+1)$-cofaces has no larger coface. Then the removal of $x$
would lead to the removal of this $(k_0+1)$-simplex with no
coface. This would create a $k_0$-th dimensional hole, and
$\beta_{k_0}$ would be incremented.
\end{proof}

\subsubsection{Bounds for the number of removed 0-simplices}
With the reduction algorithm, we go from an initial abstract
simplicial complex to a final complex with an optimal number of
$0$-simplices. We are now interested in the number of $0$-simplices
that can be removed from the initial complex.
\begin{theorem}
Let $E_k$ be the set of $0$-simplices that have index $k$ in the
initial complex, and $|E_k|$ be its cardinality. Then the number $M$
of removed $0$-simplices by the algorithm defined in Alg.~\ref{alg_ra}
is bounded by:
 \begin{eqnarray*}
    \sum_{k=k_0+2}^{I_{\max}} \mathds{1}_{E_k \neq \emptyset} \le M \le
   \sum_{k=k_0+2}^{I_{\max}} | E_k |.
  \end{eqnarray*}
\end{theorem}

\begin{proof}
We begin by looking at the upper bound. First, let us state that a
$0$-simplex with an index equal or less than $k_0+1$, can not have its
index increased during the algorithm. It is a direct consequence of
Lemma~\ref{lemma_I} and the fact that the stopping limit is when the
maximum index is equal to $k_0+1$. Then $0$-simplices of index equal
or less than $k_0+1$ are never removed by the algorithm. Thus the
number of other $0$-simplices is an upper bound for $M$.

For the lower bound,  according to Lemma~\ref{lemma_I}, the
removal of a $0$-simplex of index $I_{\max}$ can only modify the
indices that were set to $I_{\max}$ previously. In the worst case, all
indices $I_{\max}$ change and the value of $I_{\max}$ is decreased. 
Thus, at least one $0$-simplex per index value is removed. And the
number of index values strictly above $k_0+1$ constitutes a lower
bound for $M$. 
\end{proof}

\begin{remark}
We can note that the upper bound for the number of removed
$0$-simplices is optimal. Since it is reached for example in the case
of an abstract simplicial complex limited to a $n$-simplex and its
faces, with $n$ integer. Whatever the definition of the limit with the 
critical $0$-simplices, only the latter will stay in the final complex.

  We can also note that the lower bound for the number of removed 
$0$-simplices is optimal too. Indeed, it is reached for instance for a
complex limited to a $n$-simplex with $n-1$ critical $0$-simplices.
\end{remark}

\section{Complexity}
\label{sec_comp}

In this section, we investigate the complexity of the algorithm
presented in Alg.~\ref{alg_ra} for the conservation of the $k_0$-th
homology. For $S$ an abstract simplicial complex, let
us denote by $s_k$ the numbers of its $k$-simplices for any $k$
integer. We also denote by $K$ the integer such that the maximum
simplex in $S$ is a $K$-simplex. Then the size of the input data of
the reduction algorithm depends on $s_0, \dots, s_K$. 

To compute the complexity of the whole reduction algorithm, we must
first compute the complexity of the computation of the $k_0+1$ first
Betti numbers denoted $\beta_0,\dots, \beta_{k_0}$.
\begin{proposition}
The complexity of the computation of the Betti numbers $\beta_0,\dots,
\beta_{k_0}$ is in $O(\max_{k=0,\dots,k_0+1}(s_{k}^3))$.
\end{proposition}
\begin{proof}
The computation of the Betti number $\beta_k$ relies on the
computation of the ranks of matrices $\partial_k$ of size
$s_{k-1}\times s_k$ and $\partial_{k+1}$ of size $s_k\times s_{k+1}$.

Moreover, the computation of the rank of a matrix of size $n\times m$
is of complexity $O(nm\min(n,m))$.
\end{proof}

Now, we look at the complexities of the computations of the degrees
and the indices needed in the reduction algorithm:
\begin{proposition}
  The complexity of the computation of the degree of a $(k_0+1)$-simplex is
  in $O(\sum_{k=k_0+2}^K s_k)$.
\end{proposition}
\begin{proof}
To compute the degree of a $k$-simplex, we must explore at most all
the larger simplices than $k$-simplices.
\end{proof}
\begin{proposition}
  The complexity of the computation of the index of a $0$-simplex is
  in $O(s_{k_0+1})$.
\end{proposition}
\begin{proof}
The index of a $0$-simplex is just the minimum of its $(k_0+1)$-cofaces'
degrees, that are at most $s_{k_0+1}$.
\end{proof}

Then, we are able to write the whole reduction algorithm complexity:
\begin{theorem}[Reduction algorithm complexity]
The reduction algorithm described in Alg.~\ref{alg_ra} has a
complexity in:
\begin{eqnarray*}
O(s_0(\max_{0\leq k \leq k_0+1}s_{k}^3+s_{k_0+1}\sum_{k=k_0+2}^K s_k)).
\end{eqnarray*}
\label{thm_comp}
\end{theorem}
\begin{proof}
In the reduction algorithm, the $k_0+1$ first Betti numbers $\beta_0,\dots,
\beta_{k_0}$, the $s_{k_0+1}$ degrees and the $s_0$ indices are
computed at each run that is at most $s_0$ times. The computation of
the indices becomes negligible compared to the computation of the
Betti numbers.

At the end, at most every simplex has been removed, then the removal
of simplices is of overall complexity of $O(\sum_{k=0}^K s_k)$. This
complexity is negligible for the whole algorithm.

The complexity to mark $0$-simplexes as critical and the complexity to
compute $I_{\max}$ are both in $O(s_0)$, and are also negligible.
\end{proof}

\begin{remark}
The complexity of the reduction algorithm is polynomial relatively to the size of
the input data $s_0,\dots, s_K$.
\end{remark}

Traditionally we prefer to express the complexity of an algorithm just
relatively to the number of points, that is $s_0$ the number of
$0$-simplices. But we can see that the number $K$ is key here.
Indeed every number $s_k$ appears directly in the complexity
formula. However, $K$ appears as the limit of a sum, and since it is
upper bounded only by $s_0$, its behavior determines if the complexity
is polynomial or exponential in $s_0$. 

Thus, we first keep the size of the largest simplex $K$ as a variable
in the complexity formula:
\begin{corollary}
The reduction algorithm described in Alg.~\ref{alg_ra} has a
complexity upper bounded by:
\begin{eqnarray*}
O(s_0^{3k_0+7}+s_0^{k_0+3}\sum_{k=k_0+2}^K\binom{s_0}{k+1}).
\end{eqnarray*}
\label{corol_comp}
\end{corollary}
\begin{proof}
The number of $k$-simplices $s_k$ is upper bounded by $\binom{s_0}{k+1}$
for any $k$ integer. Taking the expression of the complexity from
Theorem~\ref{thm_comp}, we have:
\begin{eqnarray*}
&&s_0(\max_{0\leq k \leq k_0+1}s_{k}^3+s_{k_0+1}\sum_{k=k_0+2}^{K}s_k)\\
&\leq& s_0(\max_{0 \leq k \leq k_0+1}\binom{s_0}{k+1}^3+
       \binom{s_0}{k_0+2}\sum_{k=k_0+2}^K\binom{s_0}{k+1})\\
&\leq& s_0 (\binom{s_0}{k_0+2}^3+\binom{s_0}{k_0+2}\sum_{k=k_0+2}^K\binom{s_0}{k+1})\\
\end{eqnarray*}
Since $k_0$ is a fixed small number compared to $s_0$ when $s_0$ goes
to infinity, and $\binom{s_0}{k}$ is increasing for $k\leq k_0+2$. 

Then the fact that $\binom{s_0}{k}\leq s_0^k$ concludes the proof
\end{proof}

Then we express the complexity with only the number of $0$-simplices
$s_0$ as a parameter:
\begin{corollary}
The reduction algorithm described in Alg.~\ref{alg_ra} has a
complexity upper bounded by:
\begin{eqnarray*}
O(s_0^{k_0+3}2^{s_0}).
\end{eqnarray*}
\label{corol_comp_max}
\end{corollary}
\begin{proof}
The size of the largest simplex $K$ can only be upper bounded by $s_0$
in the general case. Then the fact that $\sum_{k=0}^n
\binom{n}{k}=2^n$ concludes the proof.
\end{proof}

We can see that the complexity of the reduction algorithm relatively
to $s_0$ can vary drastically. Indeed if the number of large simplices
is negligible relatively to the number of smaller simplices then the
computation of the Betti numbers will be the preponderant part and the
complexity of the reduction algorithm will be polynomial in
$s_0$. However, if it is there are large simplices, then the
computation of the degrees will be the longest, and  the reduction
algorithm will have a complexity that is exponential in $s_0$. 

We propose in the following of this section a thorough study on the
behavior of the complex and the size of its largest simplex $K$
that will determine the complexity of the reduction algorithm.

\subsection{Model}
In order to study the behavior of an abstract simplicial complex, we
first must choose a particular complex. We select the Vietoris-Rips
complex for its mathematical tractability and its wireless network
representation capacity. We now have to decide on which set of points
the complex is build.

When evaluating the behavior of a complex, we need to take into
account side effect: indeed a point in the center of the plane area
considered will have more neighbors, and thus be part of larger
simplices that a point on the edge of the area. In order to avoid
these side effects, we choose to consider a torus instead of a plane. 
Let us denote by $\mathbb{T}^d_a$ the torus of side $a$ in
dimension $d$. Usually, we will have $d=2$.

We now define the space of configurations:
\begin{definition}
The space of configurations on $\mathbb{T}^d_a$ is the set of locally
finite simple point measures:
\begin{equation*}
  \Omega^{X}=\left\{\omega=\sum_{k=1}^n\delta(x_k)\ :\
    {(x_k)}_{k=1}^{k=n} \subset X,\
    n\in\mathbb{N}\cup\{\infty\}\right\}, 
\end{equation*}
where $\delta(x)$ denotes the Dirac measure for $x\in \mathbb{T}^d_a$.
\end{definition}

 It is convenient to identify an element $\omega$ of $\Omega^X$ with
the set corresponding to its support, i.e. $\sum_{k=1}^n\delta(x_k)$
is identified with the unordered set $\{x_1, \dots, x_n\}$. 
For $A \subset X$, we have $\delta(x)(A)=\mathds{1}_A(x)$, so
that $\omega(A)= \sum_{x\in\omega}\mathds{1}_A(x)$ counts the number
of points in $A$.  Simple measure means that there are no two points
in the same place, that is $\omega(\{x\})\le 1$ for any $x\in X$. Locally
finite means that $\omega(K)<\infty$ for any compact $K$ of $X$. The
configuration space $\Omega^{X}$ is endowed with the vague topology
and its associated $\sigma$-algebra denoted by $\mathcal{F}^{X}$. For
further reading on random point processes, we refer to
\cite{daley_introduction_2003}.   

We now define the binomial point process that is a variation of the
well-known Poisson point process but with a fixed number of points:
\begin{definition}[Binomial point process]
Let $f$ be the uniform probability density function on the torus
$\mathbb{T}^d_a$, and $n$ an integer. Then a point process $\omega$ is
a binomial point process of $n$ points on $\mathbb{T}^d_a$, if the
following two conditions hold:
\begin{enumerate}[(i)]
\item The process $\omega$ has $n$ points,
\item The points' positions are drawn according to $f$ independently
 from each other.
\end{enumerate}
\end{definition}

In order to have reasonable results we need to make two
assumptions. First, we need to ensure that the ratio $\frac{r}{a}$ that
is the ratio between the connexion distance of the Vietoris-Rips
complex and the size of the torus, is not too big so that two points
are able to be connected $2$ times on both sides of the torus:
\begin{assumption}
  Let us denote by $\theta$ the ratio ${\left( \frac{r}{a} \right)}^d$,
  then we assume that:
  \begin{eqnarray*}
\theta={\left( \frac{r}{a} \right) }^d \leq {\left( \frac{1}{2}
  \right) }^d.
  \end{eqnarray*}
\end{assumption}

Then the behavior of the number of simplices is not easy to obtain. In
\cite{decreusefond_simplicial_2014}, the authors provide expressions
for the moments of the number of simplices for the Vietoris-Rips
complex by means of Malliavin calculus. They obtain results for the
classic Euclidean norm, however the expressions are not
tractable. That is why we make Assumption 2:
\begin{assumption}
  For the construction of the Vietoris-Rips complex based on a
  binomial point process, we use the uniform norm. 

For $x \in \mathbb{T}^d_a$ of elements $(x_1,\dots,x_d)$ the uniform
norm of $x$ is:
  \begin{eqnarray*}
    \| x\|_\infty =\max\{ |x_1|,\dots , |x_d|\}.
  \end{eqnarray*}
\end{assumption}

Then we can use the results presented in
\cite{decreusefond_simplicial_2014}: 
\begin{theorem}[\cite{decreusefond_simplicial_2014}]
  Let $k \geq 1$ be an integer. The expectation and variance of the
  number of   $k$-simplices in a Vietoris-Rips complex based on a
  binomial point   process of $n=s_0$ points on the torus
  $\mathbb{T}_a^d$ are: 
\begin{eqnarray}
 \E{s_{k-1}}&=&\binom{n}{k} \theta^{k-1} k^d \label{chimean_e}\\
\var{s_{k-1}}&=&\sum_{i=1}^{k+1} 
\binom{n}{2k-i} \binom{2k-i}{k} \binom{k}{i} \nonumber\\
&&\theta^{2k-i-1}
{\left( 2k-i+2\frac{{(k-i)}^2}{i+1} \right)}^d.
\label{chimean_var}
\end{eqnarray}
\label{theorem_chimean}
\end{theorem}

For a better reading we will denote by $n$ the number of points of the
binomial point process, that is equal to $s_0$ the number of
$0$-simplices of the abstract simplicial complex.
Throughout this section we will investigate the almost sure asymptotic
behavior of the size of the largest simplex $K$ and the complexity of
the reduction algorithm when $n$ tends to infinity and with respect
to the distribution of the $n$ points according to a binomial point
process. 
\begin{definition}
We say that the property $P$ is true  asymptotically almost surely if
$\P[P \text{ true}] \rightarrow 1$ when $n$ tends to infinity.
\end{definition}

\subsection{Percolation regimes}
One can easily see that the number $K$ that is the size of the largest simplex
in a complex is the equivalent of the clique number $C$ in a
graph. More precisely we have that $K=C-1$, since a $k$-simplex has
$k+1$ points. Moreover the Vietoris-Rips complex is by definition the
clique complex of the geometric graph. Then the Vietoris-Rips complex
based on a binomial point process is the clique complex of the random
geometric graph. That is why we use the same percolation regimes to
study $K$ as the ones described for random geometric graphs. For
further reading random geometric graphs and the definition of
percolation regimes, see~\cite{penrose_random_2003}.
We show in Figure~\ref{fig_regimes} the three different percolation regimes.

\begin{figure}[h]
  \centering
  \includegraphics[width=2.8cm]{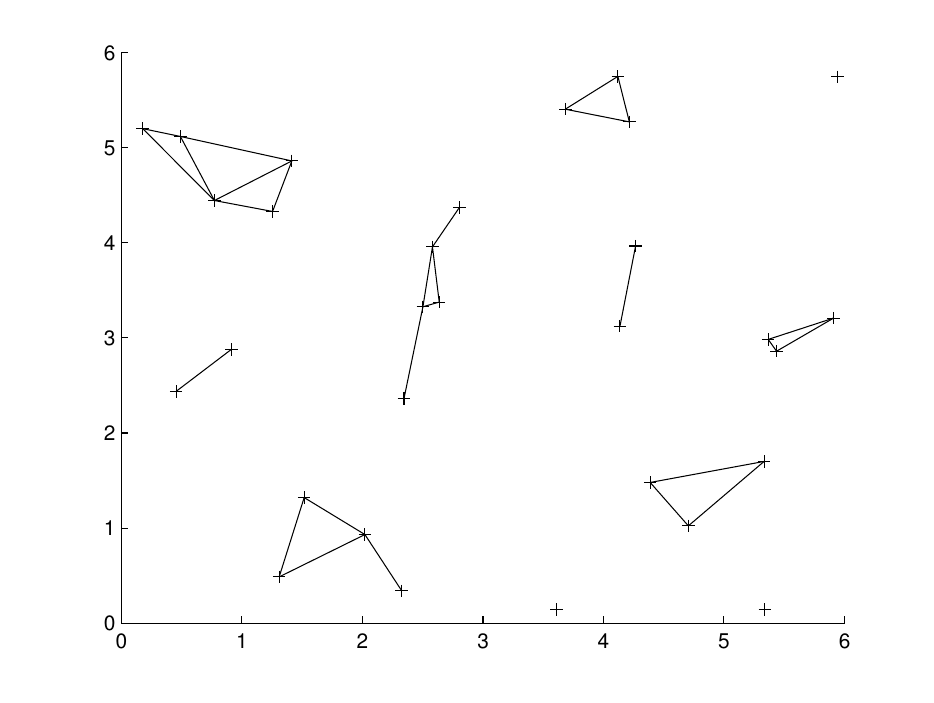}
\includegraphics[width=2.8cm]{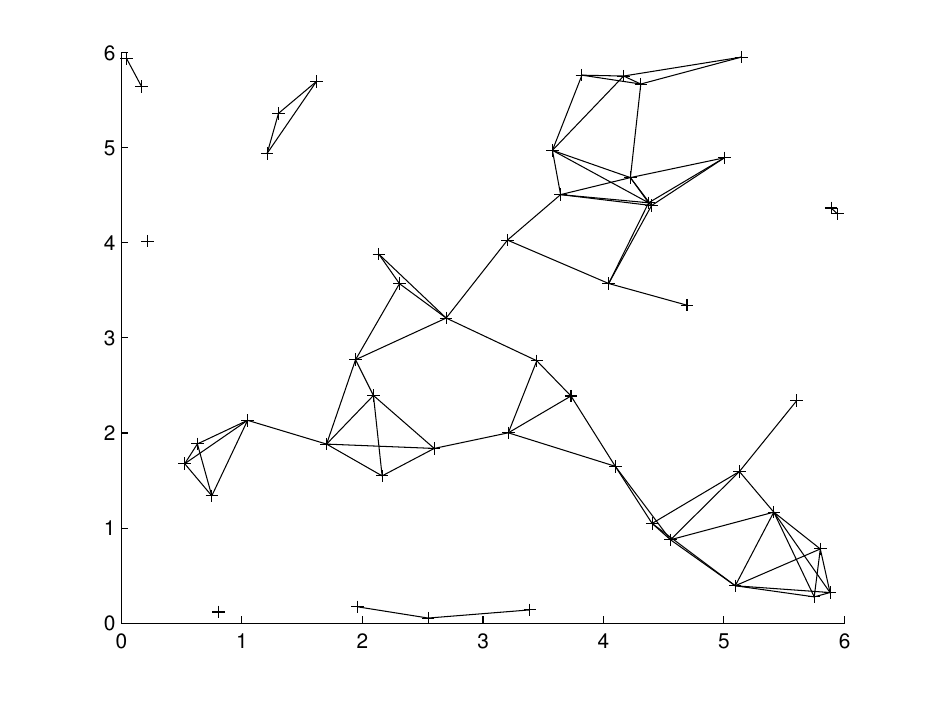}
\includegraphics[width=2.8cm]{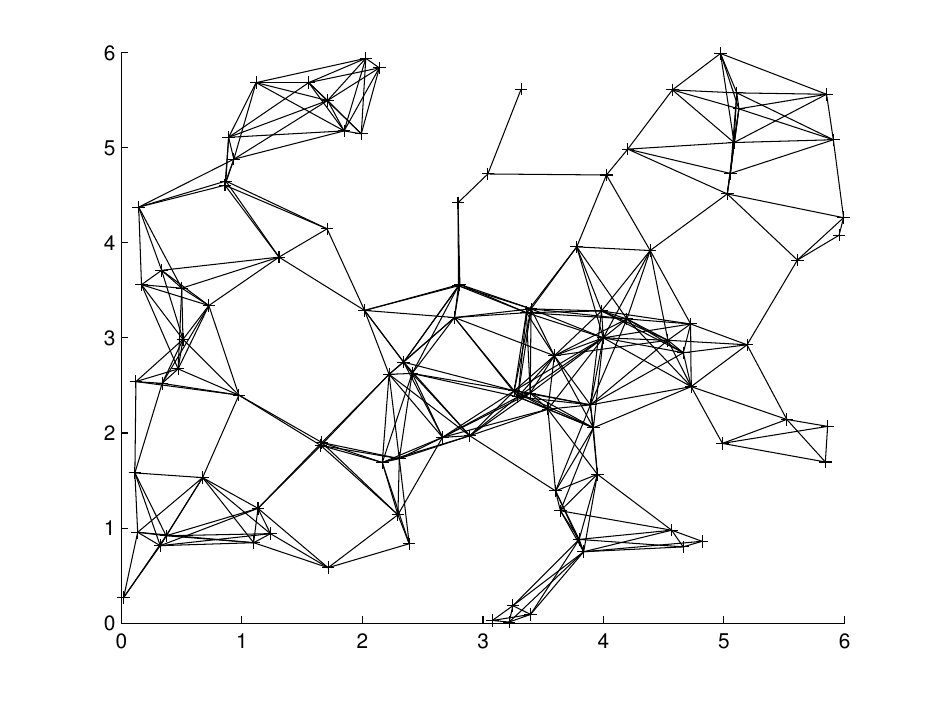}
  \caption{Percolation regimes}
\label{fig_regimes}
\end{figure}

The percolation regimes are defined with the use of the Bachman-Landau
notations that we define here. For non-negative functions $f$ and $g$
we write as $n$ tends to infinity:
\begin{itemize}
\item $f(n)=o(g(n))$ if for every $\varepsilon >0$ there exists $M$
  such that for $n \geq M$, we have $f(n)\leq \varepsilon
  g(n)$. We say that $f$ is dominated by $g$ asymptotically.
\item $f(n)=O(g(n))$ if there exists $k>0$ and $M$ such that for $n
  \geq M$, we have $f(n) \leq k g(n)$. We say that $f$ is
  bounded by $g$ asymptotically.
\item $f(n) \sim g(n)$ if $f(n)=O(g(n))$ and $g(n)=O(f(n))$. We say
  that $f$ and $g$ are equal asymptotically.
\item $f(n) \ll g(n)$ if $\frac{f(n)}{g(n)}=o(1)$. We say that $f$ is
  small compared to $g$ asymptotically. 
\end{itemize}

\subsection{Subcritical regime}
In this subsection, we consider that $\theta=o(\frac{1}{n})$. In the
subcritical regime, the Vietoris-Rips complex is composed of various
disconnected components. We can see an example of a Vietoris-Rips
complex in the subcritical regime in Figure~\ref{fig_subcrit}.

\begin{figure}[h]
  \centering
    \includegraphics[width=8.5cm]{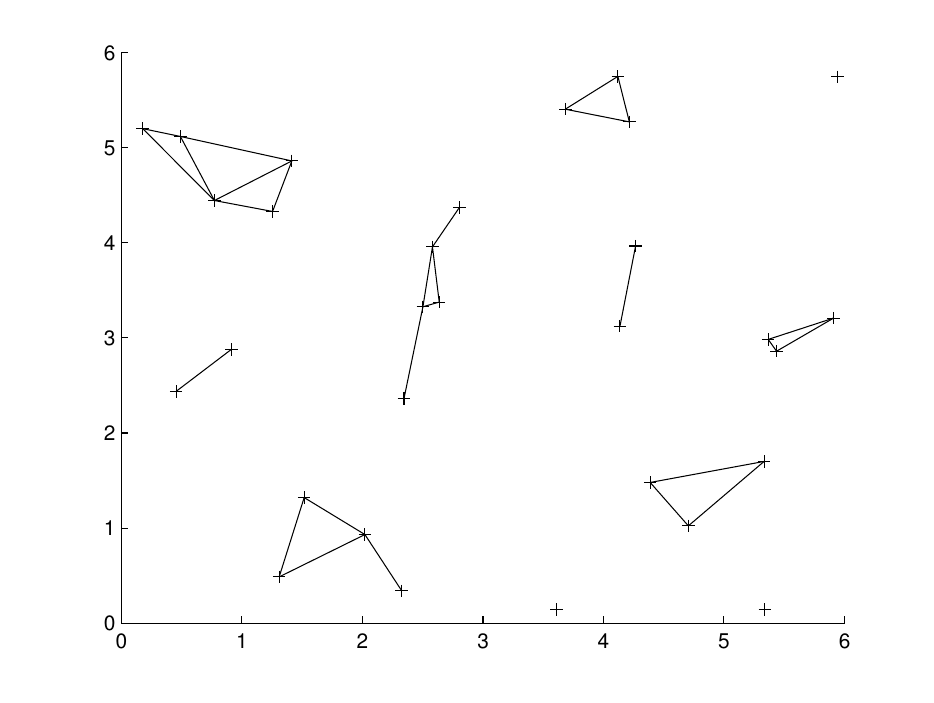}
  \caption{Subcritical regime $\theta=o(\frac{1}{n})$}
\label{fig_subcrit}
\end{figure}

Therefore the size of the largest simplex $K$
is expected to be rather small, $n$-simplices do not exist in this
regime. That is why we focus on the number of
$k$-simplices $s_k$ for $k \geq 1$ small compared to $n$, i.e. $k \ll n$.

We are able to derivate approximations from Theorem
\ref{theorem_chimean}:
\begin{lemma}
For $k \geq 1$ and $k \ll n$ in the subcritical regime,
\begin{eqnarray*}
\E{s_k} &\sim& \frac{n^{k+1}}{(k+1)!} \theta^{k}{(k+1)}^d\\
\var{s_k} &\sim& \frac{n^{k+1}}{(k+1)!} \theta^{k} {(k+1)}^d.
\end{eqnarray*}
\label{lemma_subcrit}
\end{lemma}
\begin{proof}
This is a direct consequence of the subcritical regime hypothesis applied
to Equations~\ref{chimean_e} and~\ref{chimean_var}.
\end{proof}

Meanwhile, in the subcritical regime, the random
geometric graph shares similar properties with the Erd\"os-R\'enyi
model, that is the graph with $n$ points where each edge is chosen
independently with probability $p$. The clique number of the
Erd\"os-R\'enyi graph has been studied first in
\cite{matula_complete_1970}, and its almost sure behavior has been
described in~\cite{bollobas_cliques_1976}: 
\begin{theorem}[\cite{bollobas_cliques_1976}]
Let us define $n_k \sim p^{-k/2}$ and $n_k' \sim (1+\frac{3\log
  k}{k})p^{-k/2}$.
When $n$ goes to infinity, with a fixed probability $p$, for almost
every graph, there is a constant $c$ such that if $n_k'\leq n \leq
n_{k+1}$ for some $k>c$, then the clique number is $K=k$.
\end{theorem}

 That means that the size of the largest simplex $K$ grows
slowly step by step as the number of points $n$ goes to infinity.
We find similar results for the Vietoris-Rips complex:
\begin{theorem}
Let $k$ be a non zero integer and $\eta$ a strictly positive real
number. We then define:
\begin{eqnarray*}
\theta_k'&=&{\left(\frac{(k+1)! {(k+1)}^{1+\eta-d}}{n^{k+1}}\right)}^{1/k}\\
\theta_k&=&{\left(\frac{(k+1)!{(k+1)}^{-(1+\eta+d)}}{n^{k+1}}\right)}^{1/k}.
\end{eqnarray*}
Then for $k \ll n$ and $\theta_{k}'<\theta<\theta_{k+1}$ in the
subcritical regime, the size of the largest simplex in a Vietoris-Rips
complex based on a binomial point process of $n$ points on the torus
$\mathbb{T}^d_a$ is asymptotically almost surely: 
\begin{eqnarray*}
K=k.
\end{eqnarray*}
\label{thm_sub}
\end{theorem}
\begin{proof}
For $\theta>\theta_k'$, thanks to the approximations of Lemma
\ref{lemma_subcrit} we have:
\begin{eqnarray*}
\E{s_k} \geq  \frac{n^{k+1}}{(k+1)!} \theta_k'^{k}{(k+1)}^d
 = {(k+1)}^{1+\eta}.
\end{eqnarray*}
And for  $\theta<\theta_k$,
\begin{eqnarray*}
\E{s_k} \leq  n^{k+1} \theta_k^{k}{(k+1)}^d
 = {(k+1)}^{-(1+\eta)}.
\end{eqnarray*}

Moreover, for $k \ll n$ when $n$ goes to infinity, we consider $n$
large enough so that $n>k! k^{k(1+2\eta)}$. Since $k^k>{(k+1)}^{k-1}$ for
all $k\geq 1$, we have: 
\begin{eqnarray*}
 && n>k^{k(1+2\eta)}k! \\
&\Rightarrow& n > k^{k(1+\eta-d)}k^{k(\eta+d)}k!\\
&\Rightarrow& n > k^{k(1+\eta-d)}{(k+1)}^{(k-1)(\eta+d)}k!\\
&\Rightarrow& n > k^{k(1+\eta-d)}{(k+1)}^{(k-1)(1+\eta+d)}\frac{1}{{(k+1)}^{k-1}}k!\\
&\Rightarrow& n > k^{k(1+\eta-d)}{(k+1)}^{(k-1)(1+\eta+d)}\frac{{(k!)}^k}{{((k+1)!)}^{k-1}}\\
&\Rightarrow& n^{\frac{1}{k(k-1)}} >\frac{{(k^{1+\eta-d}k!)}^{1/(k-1)}}{{({(k+1)}^{-(1+\eta+d)}(k+1)!)}^{1/k}}\\
&\Rightarrow& \frac{n^{\frac{k}{k-1}}}{n^{\frac{k+1}{k}}}> 
              \frac{{(k^{1+\eta-d}k!)}^{1/(k-1)}}{{({(k+1)}^{-(1+\eta+d)}(k+1)!)}^{1/k}}\\
&\Rightarrow&\theta_k>\theta_{k-1}'.
\end{eqnarray*}

We can now consider $\theta$ such that
$\theta_k'<\theta<\theta_{k+1}$.

On the one hand, we find an upper bound for the probability of
the non-existence of $k$-simplices:
\begin{eqnarray*}
\P[s_k=0,\theta>\theta_k'] \leq \frac{\var{s_k}}{\E{s_k}^2}
\sim \frac{1}{\E{s_k}}
\leq \frac{1}{{(k+1)}^{1+\eta}}.
\end{eqnarray*}

On the other hand, we can find an upper bound for the probability of existence
of $(k+1)$-simplices:
\begin{eqnarray*}
\P[s_{k+1}>0, \theta<\theta_{k+1}] \leq \E{s_{k+1}}
\leq \frac{1}{{(k+2)}^{1+\eta}}.
\end{eqnarray*}

Finally we have: 
\begin{align*}
\P[\exists \theta, \theta_k'<\theta<\theta_{k+1}, K&\neq k]\\
&<\frac{1}{{(k+1)}^{1+\eta}}+\frac{1}{{(k+2)}^{1+\eta}}.
\end{align*}

 As the sum $\sum_{k=1}^{\infty}k^{-(1+\eta)}$ converges, the
 Borel-Cantelli theorem implies that with the exception of finitely
 many $k$'s, for all $\theta$ such that
 $\theta_k'<\theta<\theta_{k+1}$, one has $K=k$. Then when $n$ goes to
 infinity, we have asymptotically almost surely that $K=k$ as $n$ goes
 to infinity:
\begin{eqnarray*}
\P[K= k, \theta_k'<\theta<\theta_{k+1}] \xrightarrow{n \rightarrow
  \infty} 1,
\end{eqnarray*}
concluding the proof.
\end{proof}

Finally we can conclude on the complexity of the reduction algorithm
in this regime:
\begin{theorem}
  The reduction algorithm described in Alg.~\ref{alg_ra} for a
  Vietoris-Rips complex based on a binomial point process of $n$
  points on the torus $\mathbb{T}^d_a$ in the subcritical regime such
  that $\theta=o(\frac{1}{n})$, has a complexity in:
  \begin{eqnarray*}
    O(n^{3k_0+7}).
  \end{eqnarray*}
\end{theorem}
\begin{proof}
  This is a direct consequence of Corollary~\ref{corol_comp} and
  Theorem~\ref{thm_sub} as the first part of the complexity becomes
  preponderent. 
\end{proof}

\subsection{Critical regime}
We now consider that $\theta \sim \frac{1}{n}$ so that the abstract
simplicial complex is in the critical regime. In this regime,
percolation occurs: the disconnected components begin to connect into
one sole connected component. We can see an example of a Vietoris-Rips
complex in the critical regime in Figure~\ref{fig_crit}.

\begin{figure}[H]
  \centering
    \includegraphics[width=8.5cm]{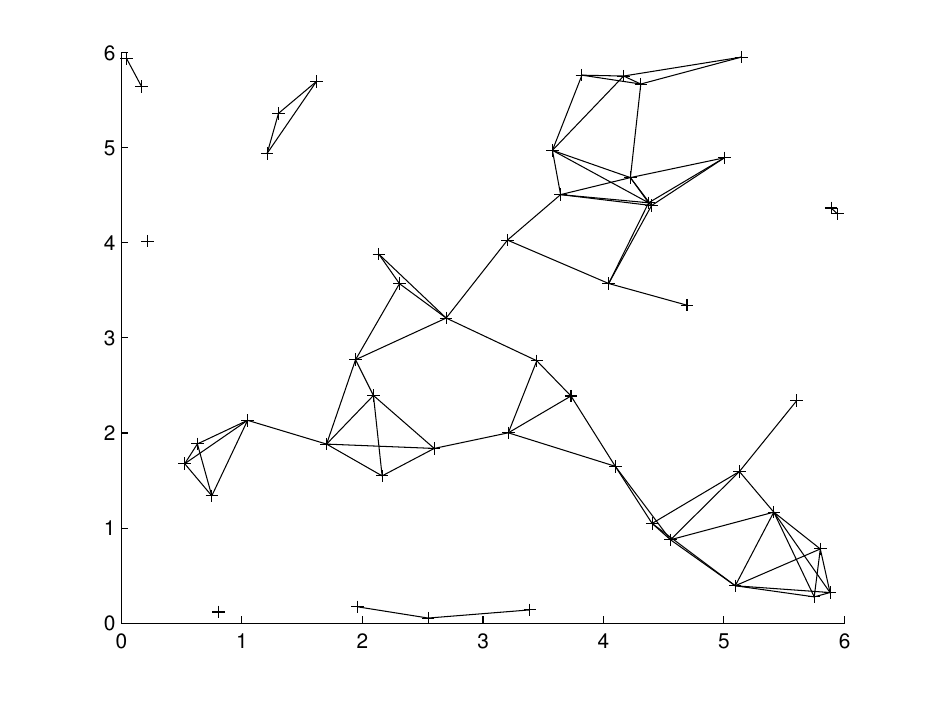}
  \caption{Critical regime $\theta \sim \frac{1}{n}$}
\label{fig_crit}
\end{figure}

The size of the largest simplex $K$ is still rather small compared to
$n$ as $n$ goes to infinity, $n$-simplices still do not exist in this
regime. This allows us to focus on the number of $k$-simplices $s_k$
for $k \geq 1$ bounded by $n$ asymptotically, i.e. $k =O(n)$.

We are able to derivate new approximations from Theorem
\ref{theorem_chimean} in this regime:
\begin{lemma}
For $k \geq 1$ and $k =O(n)$ in the critical regime,
\begin{eqnarray*}
\E{s_k} &\sim& \frac{1}{\sqrt{2\pi}}n{(k+1)}^{d-k+\frac{1}{2}}\\
\var{s_k} &\sim& \frac{1}{\sqrt{2\pi}}n{(k+1)}^{d-k+\frac{1}{2}}.
\end{eqnarray*}
\label{lemma_crit}
\end{lemma}
\begin{proof}
The expectation approximation is a direct consequence of Equation
\ref{chimean_e}, the critical regime approximation $\theta \sim
\frac{1}{n}$ and Stirling's approximation: $n! \sim
\sqrt{2 \pi n}{\left( \frac{n}{e} \right)}^n$.

Then, since $k=O(n)$, we can approximate the variance by its
dominating term in Equation~\ref{chimean_var} that is $i=k$ that leads
to $\var{s_k}\sim \E{s_k}$.
\end{proof}

We then derive from these approximations, the almost sure asymptotical
behavior of the size of the largest simplex $K$:
\begin{theorem}
In the critical regime, the size of the largest simplex in a
Vietoris-Rips complex based on a binomial point process of $n$ points
on the torus $\mathbb{T}^d_a$, grows asymptotically almost surely
slower than $\ln n$ with an arbitrary small distance. That means that
for all $\eta$ strictly positive real number:  
\begin{eqnarray*}
 {(\ln n)}^{1-\eta} < K < \ln n.
\end{eqnarray*}
\label{thm_crit}
\end{theorem}
\begin{proof}
First for $k>\ln n$, we find an upper bound for the expectation of the
number of $k$-simplices thanks to the approximation of Lemma
\ref{lemma_crit}:
\begin{eqnarray*}
\E{s_k}<\frac{1}{\sqrt{2\pi}}n{(\ln n+1)}^{d+\frac{1}{2}-\ln n}.
\end{eqnarray*}
One can easily check that this upper bound tends to $0$ as $n$ goes to
infinity. Then, since $\P[s_k>0] \leq \E{s_k}$, the probability that there exists
$k$-simplices tends to $0$:
\begin{eqnarray*}
\P[K>k]=\P[s_k>0] \xrightarrow{n \rightarrow \infty} 0 \quad \forall k>\ln n,
\end{eqnarray*}
and $K<\ln n$ asymptotically almost surely. 

On the other hand, for $k<{(\ln n)}^{1-\eta}$ for all $\eta>0$, we find
a lower bound for the expectation of the number of $k$-simplices: 
\begin{eqnarray*}
\E{s_k}>\frac{1}{\sqrt{2\pi}}  n{({(\ln n)}^{1-\eta}+1)}^{d+\frac{1}{2}-{(\ln n)}^{1-\eta}}. 
\end{eqnarray*}
This lower bound tends to infinity as $n$ grows. Thanks to the
asymptotic equivalence of the variance and the expectation of the
number of $k$-simplices, we have that $\P[s_k=0] \leq
\frac{1}{\E{s_k}}$, and the probability that there exists no
$k$-simplices tends to infinity:
\begin{eqnarray*}
\P[K<k]=\P[s_k=0] \xrightarrow{n \rightarrow \infty} 0 \quad \forall k<{(\ln n)}^{1-\eta},
\end{eqnarray*}
and $K>{(\ln n)}^{1-\eta}$ asymptotically almost surely. 
\end{proof}

We can now derive the complexity of the reduction algorithm in the
critical regime:
\begin{theorem}
  The reduction algorithm described in Alg.~\ref{alg_ra} for a
  Vietoris-Rips complex based on a binomial point process of $n$
  points on the torus $\mathbb{T}^d_a$ in the critical regime such
  that $\theta\sim \frac{1}{n}$ has a complexity in:
  \begin{eqnarray*}
    O(n^{\ln n}).
  \end{eqnarray*}
\end{theorem}
\begin{proof}
  The complexity from Corollary~\ref{corol_comp} is
  $O(n^{3k_0+7}+n^{k_0+3}\sum_{k=k_0+1}^K\binom{n}{k+1})$. Then
  Theorem~\ref{thm_crit} gives us an approximation for $K$.

Since $\ln n < \frac{n}{2}$ for every $n\geq 2$, the preponderent term
of the sum is $\binom{n}{\ln n}$ that can only be upper bounded by
$n^{\ln n}$.  
\end{proof}

\subsection{Supercritical regime}
In the supercritical regime, we have that
$\frac{1}{n}=o(\theta)$. Percolation has occurred: the Vietoris-Rips 
complex is now connected and tends to become the complete 
complex, I.e.\ the complex with all simplices by analogy with the
complete graph. We can see an instance of a Vietoris-Rips complex in
this regime in Figure~\ref{fig_supercrit}.

\begin{figure}[h]
  \centering
    \includegraphics[width=8.5cm]{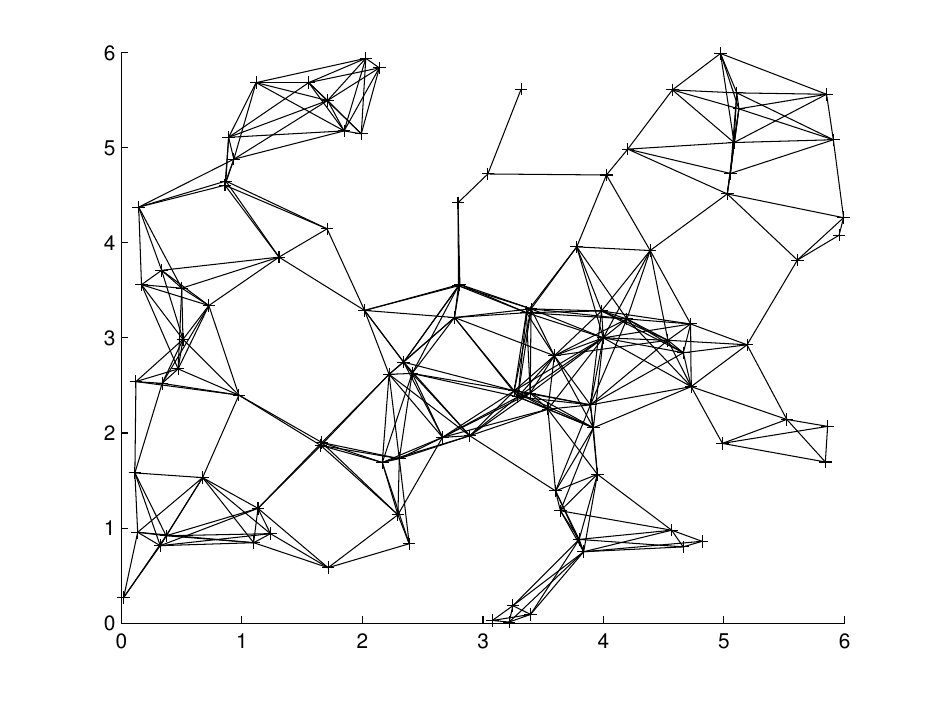}
  \caption{Supercritical regime $\frac{1}{n}=o(\theta)$}
\label{fig_supercrit}
\end{figure}

In the supercritical regime, it is no more possible to obtain
approximations from the exact formulas of Theorem
\ref{theorem_chimean} since $K$ becomes equivalent to $n$ as they
tend to infinity.  

However, the behavior of the size of the largest simplex for a
binomial point process has already be studied by Appel and Russo in
\cite{appel_maximum_1997} as the clique number of a 
random geometric graph.  They first find the almost sure asymptotic
rate for the maximum vertex degree. Then by squeezing the clique
number between two values of the maximum vertex degree, they obtain
its asymptotic behavior. We propose here an alternative approach. 

First, we state a fact true in any percolation regime:
\begin{lemma}
  In a Vietoris-Rips complex based on a binomial point process of $n$
  points we always have that:
  \begin{eqnarray*}
    n \leq (K+1)\lceil \frac{1}{\theta} \rceil.
  \end{eqnarray*}
\label{lemma_super}
\end{lemma}
\begin{proof}
  Let us consider a lattice square grid of spacing $r$, the parameter
  of the Vietoris-Rips complex, on the torus $\mathbb{T}^d_a$. Then
  the number of little squares of side $r$ is:
  \begin{eqnarray*}
    \lceil {\left( \frac{a}{r} \right)}^d \rceil = \lceil \frac{1}{\theta} \rceil.
  \end{eqnarray*}
 
All points that lie inside a same little square of side $r$ are
connected to each other, and are in the same simplex, by definition of
the Vietoris-Rips complex. The number $K$ is the size of the largest
simplex, so there are at most $K+1$ points in the same square.

The sum on all the squares concludes the proof.
\end{proof}

We can now write the main theorem for the behavior of $K$ in the
supercritical regime:
\begin{theorem}
In the supercritical regime, the size of the largest simplex in a
Vietoris-Rips complex based on a binomial point process of $n$ points
on the torus $\mathbb{T}^d_a$ grows asymptotically almost surely as
$n\theta$: 
\begin{eqnarray*}
  K \sim n\theta.
\end{eqnarray*}
\label{thm_supercrit}
\end{theorem}
\begin{proof}
  On the one hand, we know from Lemma~\ref{lemma_super} that
$n \leq (K+1)\lceil \frac{1}{\theta} \rceil $.
Then as $\lceil\frac{1}{\theta}\rceil \leq \frac{1}{\theta}+1$, we
have that: 
\begin{eqnarray*}
  K\geq \frac{n\theta}{1+\theta}.
\end{eqnarray*}
As $n$ goes to infinity, $\theta$ tends to $0$. Then we have
asymptotically almost surely that:
\begin{eqnarray*}
  K\geq n\theta.
\end{eqnarray*}

On the other hand, by definition of the Vietoris-Rips complex, a
$k$-simplex occurs when $k+1$ points are in the same ball of diameter
$r$. Without loss of generality, we can center the ball on one of the
point. Then we can write:
\begin{eqnarray*}
  \P[K>n\theta] &=&  \P[s_{n\theta}>0]\\
&=& \P[\exists \ x, |B(x,\frac{r}{2})| \geq n\theta],\\
\end{eqnarray*}
where $x$ is a point of the binomial point process, and 
$| B(x,\frac{r}{2})|$ counts the number of points of the process that
lie in the ball $B(x,\frac{r}{2})$ centered in $x$ of radius $\frac{r}{2}$.

Let $x_1,\dots, x_n$ denote the $n$ points of the binomial point
process, 
\begin{eqnarray*}
  \P[K>n\theta] 
&=&\P[\exists \ 1\leq i \leq n, | B(x_i,\frac{r}{2}) | \geq n\theta]\\
&\leq& \P[\bigcup_{i=1}^n | B(x_i,\frac{r}{2}) | \geq n\theta]\\
&\leq& \sum_{i=1}^n \P[| B(x_i,\frac{r}{2})|\geq n\theta].
\end{eqnarray*}
Then by stationarity of the binomial point process, we have:
\begin{eqnarray*}
 \P[K>n\theta]   
&\leq& \sum_{i=1}^n \P[| B(x_i,\frac{r}{2})|\geq n\theta]\\
&\leq& n \P[|B(x_1,\frac{r}{2})|\geq n\theta].
\end{eqnarray*}

The number of points in the ball $B(x_1,\frac{r}{2})$ follows a binomial
distribution of $n-1$ points and of probability $\theta$: 
$\text{Binom}(n-1,\theta)$. Therefore Hoeffding's inequality implies
that:
\begin{eqnarray*}
  \P[K>n\theta]   
&\leq& n \P[|B(x_1,\frac{r}{2})|\geq n\theta]\\
&\leq& n \P[|B(x_1,\frac{r}{2})|\geq (n-1)(\theta+\frac{\theta}{n-1})]\\
&\leq& n \exp{(-2\frac{\theta^2}{n-1})}.
\end{eqnarray*}
As this last upper bound tends to $0$ as $n$ tends to infinity, we
have asymptotically almost surely:
\begin{eqnarray*}
  K\leq n\theta.
\end{eqnarray*}
\end{proof}

We are now able to derive the complexity of the reduction algorithm in
the supercritical regime:
\begin{theorem}
  The reduction algorithm described in Alg.~\ref{alg_ra} for a
  Vietoris-Rips complex based on a binomial point process of $n$
  points on the torus $\mathbb{T}^d_a$ in the supercritical regime
  such that $\frac{1}{n}=o(\theta)$ has a complexity in:
  \begin{eqnarray*}
    O(n^{k_0+3}2^n).
  \end{eqnarray*}
\end{theorem}
\begin{proof}
This is Corollary~\ref{corol_comp_max}.
\end{proof}
We can see that in this regime, we were not able to improve the
complexity via the behavior of $K$.

\section{Conclusion}
\label{sec_ccl}
In this paper, we have presented a reduction algorithm for abstract
simplicial complexes. We apply this algorithm to
Vietoris-Rips complexes that can represent wireless networks and their
topology, it then provides a solution for energy saving in redundant
wireless networks. 

We have proved that our reduction algorithm maintains the complex's
homology, works in an optimal order for computation's complexity,
and reaches an optimal solution. Finally we have investigated its
complexity depending on the size of the input, and in a second
approach depending only on the number of points.

\bibliographystyle{IEEEtran}
\bibliography{cliquebib}

\begin{IEEEbiography}[{\includegraphics[width=1in,height=1.25in,clip,keepaspectratio]
{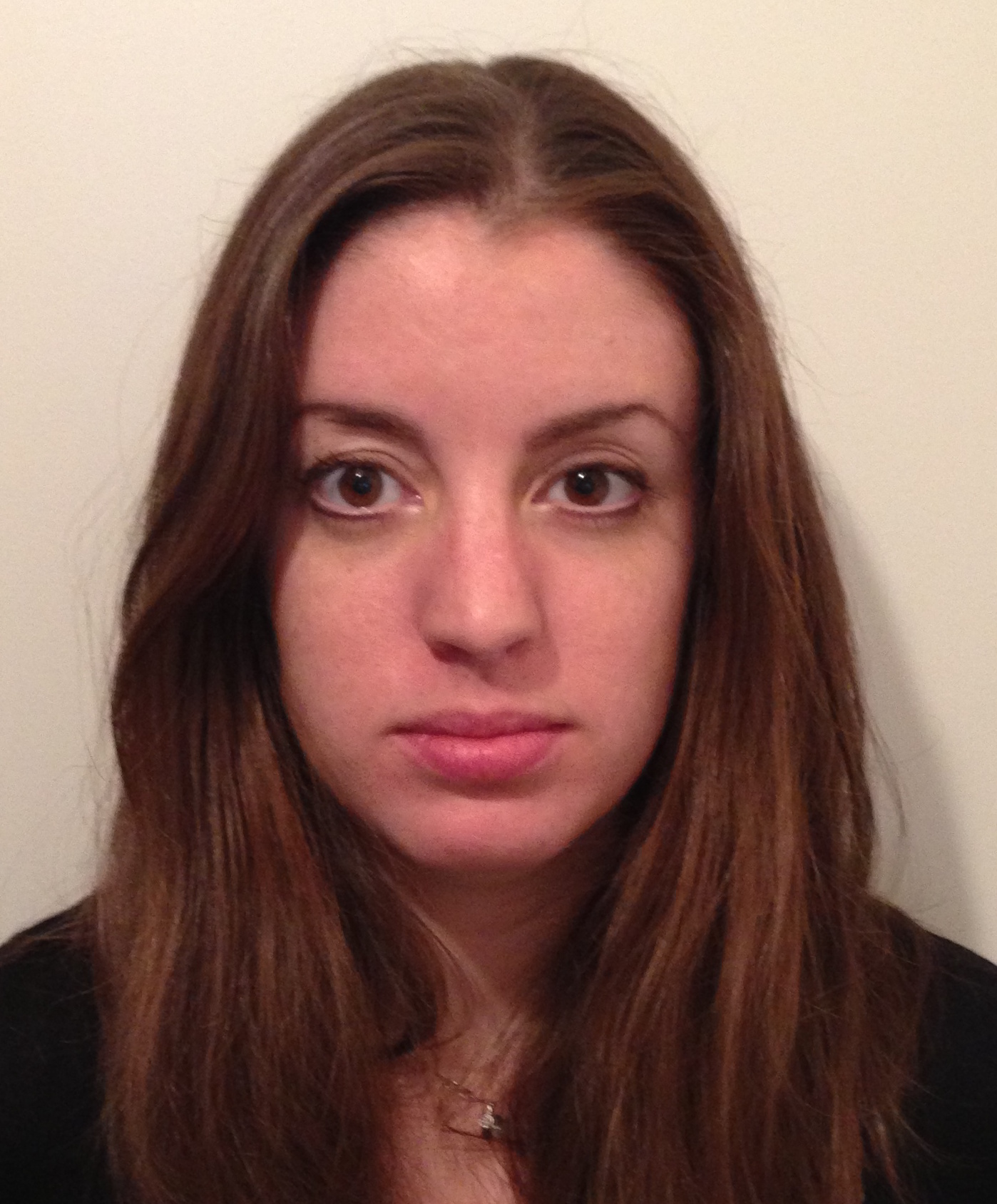}}]{Ana\"\i s Vergne}
received the Dipl.Ing.\ degree in telecommunications from T\'el\'ecom
ParisTech, Paris, France in 2010. She obtained the Ph.D. degree in
networking and computer sciences in 2013 also from T\'el\'ecom ParisTech,
Paris, France. She is currently an associate professor in the Network
and Computer Science Department at T\'el\'ecom ParisTech. Her research
interests include simplicial homology, algebraic topology, stochastic
geometry, and their applications to wireless sensor networks and
cellular networks. 
\end{IEEEbiography}

\begin{IEEEbiography}[{\includegraphics[width=1in,height=1.25in,clip,keepaspectratio]{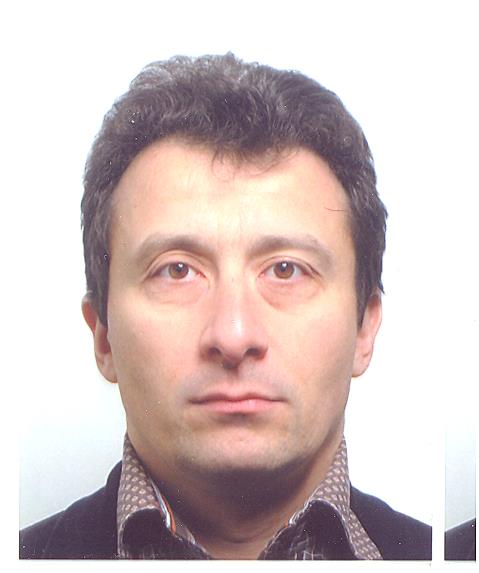}}]{Laurent Decreusefond}
is a former student of Ecole Normale Supérieure de Cachan. He obtained
his Ph.D. degree in Mathematics in 1994 from T\'el\'ecom ParisTech and his
Habilitation in 2001. He is currently a Professor in the Network and
Computer Science Department, at T\'el\'ecom ParisTech. His main fields of
interest are the Malliavin calculus, the stochastic analysis of long
range dependent processes, random geometry and topology and their
applications. With P. Moyal, he co-authored a book about the
stochastic modeling of telecommunication. 
\end{IEEEbiography}

\begin{IEEEbiography}[{\includegraphics[width=1in,height=1.25in,clip,keepaspectratio]{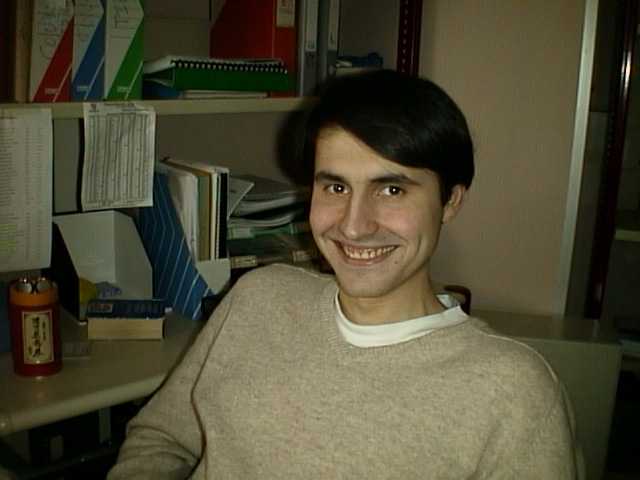}}]{Philippe Martins}
received a M.S. degree in signal processing and another M.S. degree in
networking and computer science from Orsay University and ESIGETEL
France, in 1996. He received the Ph.D. degree in electrical
engineering from T\'el\'ecom ParisTech, Paris, France, in 2000. He is
currently a Professor in the Network and Computer Science Department,
at T\'el\'ecom ParisTech. His main research interests lie in performance
evaluation in wireless networks (RRM, scheduling, handover algorithms,
radio metrology). His current investigations address mainly three
issues: a/ the design of distributed sensing algorithms for cognitive
radio b/ distributed coverage holes detection in wireless sensor
networks c/ the definition of analytical models for the planning and
the dimensioning of cellular systems. He has published several papers
on different international journals and conferences. He is also an
IEEE senior member and he is co-author of several books on 3G and 4G
systems. 
\end{IEEEbiography}

\end{document}